\documentclass[a4paper, 10pt]{amsart}

\usepackage{amsmath,amsthm,amssymb,graphicx,rotating,
pinlabel,yhmath,tikz,listings,xcolor}
\usepackage[all]{xy}

\SelectTips{cm}{}

\allowdisplaybreaks
\numberwithin{equation}{section} 

\usepackage{hyperref}
\hypersetup{colorlinks=true,citecolor=brown,
linkcolor=brown,urlcolor=black,filecolor=black}

\theoremstyle{definition}
\newtheorem{definition}{Definition}[section]
\newtheorem{remark}[definition]{Remark}
\newtheorem{example}[definition]{Example}

\theoremstyle{plain}
\newtheorem{theorem}[definition]{Theorem}
\newtheorem{proposition}[definition]{Proposition}
\newtheorem{lemma}[definition]{Lemma}
\newtheorem{corollary}[definition]{Corollary}

\newcommand{\up}{\vspace{-0.5cm}}

\newcommand{\Q}{\mathbb{Q}}
\newcommand{\Z}{\mathbb{Z}}

\newcommand{\calA}{\mathcal{A}}
\newcommand{\calI}{\mathcal{I}}

\newcommand{\calM}{\mathcal{M}}

\newcommand{\calK}{\mathcal{K}}

\newcommand{\calQ}{\mathcal{Q}}
\newcommand{\calR}{\mathcal{R}}
\newcommand{\calS}{\mathcal{S}}

\newcommand{\Gr}{\operatorname{Gr}}
\newcommand{\Hom}{\operatorname{Hom}}
\newcommand{\Lie}{\operatorname{Lie}}
\newcommand{\Push}{\operatorname{Push}}
\newcommand{\Sp}{\operatorname{Sp}}

\newcommand{\ab}{\operatorname{ab}}
\newcommand{\abf}{\operatorname{abf}}
\newcommand{\id}{\operatorname{id}}

\newcommand{\ve}{\varepsilon}

\newcommand{\ov}[1]{\widetilde{#1}}

\newcommand{\clo}[1]{\wideparen{#1}}     

\newcommand{\Com}[6]{{\rotatebox[origin=c]{180}{\rm{\large\textsf{V}}}}
\!\left({#1},{#2},{#3}\,;\,{#4},{#5},{#6}\right)}

\newcommand{\Diff}[6]{\mathsf{D}\!\left({#1},{#2}\,;\,{#3}
,{#4}\,;\,{#5},{#6}\right)}

\newcommand{\Square}[6]{\mathsf{S}\!\left({#1},{#2}\,;\!
\begin{array}{cc}{#3}\!\! & {#4}\!\!\\
{#5}\!\! & {#6}\!\! \end{array}\right)}

\newcommand{\Triple}[4]{\mathsf{T}\!\left({#1}\,;\,{#2} ,{#3},{#4}\right)}

\newcommand{\IHXone}[5]{\mathsf{IHX}_1\!\left({#1};{#2},{#3},{#4}\,;\,{#5}\right)}

\newcommand{\IHXtwo}[4]{\mathsf{IHX}_2\!\left({#1},{#2}\,;\,{#3}\,;\,{#4}\right)}

\newcommand{\IHXthree}[3]{\mathsf{IHX}_3\!\left({#1},{#2}\,;\,{#3}\right)}

\newcommand{\IHXthreep}[4]{\mathsf{IHX}'_3\!\left({#1},{#2}\,;\,{#3},{#4}\right)}

\newcommand{\by}[1]{\stackrel{\eqref{#1}}{=}}

\newcommand{\sdl}{\\[0.3cm]}

\begin{document}

\title[On the degree-two part of the associated graded 
of the Torelli group]{On the degree-two part of the associated graded \\ of  the lower central series of the Torelli group}

\author{Quentin Faes}
\address{Institute of Mathematics, University of Zurich, Winterthurerstrasse 190, CH-8057 Zurich, Switzerland}
\email{quentin.faes@math.uzh.ch}

\author{Gw\'ena\"el Massuyeau}
\address{
Universit\'e Bourgogne Europe, CNRS, IMB (UMR 5584), 21000 Dijon, France}
\email{gwenael.massuyeau@ube.fr}

\author{Masatoshi Sato}
\address{
Department of Mathematics and Data Science,
Tokyo Denki University,
5 Senjuasahi-cho, Adachi-ku, Tokyo 120-8551,
Japan}
\email{msato@mail.dendai.ac.jp}

\date{}

\thanks{}

\begin{abstract}
We consider the associated graded 
$\bigoplus_{k\geq 1} \Gamma_k \mathcal{I} / 
\Gamma_{k+1} \mathcal{I} $
of  the lower central series 
$\mathcal{I} = \Gamma_1 \mathcal{I}
\supset \Gamma_2 \mathcal{I}
\supset \Gamma_3 \mathcal{I}  \supset \cdots$
of  the Torelli group $\mathcal{I}$ 
of a compact oriented surface.
Its degree-one part is well-understood  by D. Johnson's seminal works
on the abelianization of the  Torelli group.
The knowledge of the degree-two part 
$(\Gamma_2 \mathcal{I} / \Gamma_3 \mathcal{I})\otimes \Q$
with rational coefficients
arises from works of S. Morita  on the Casson invariant
and R. Hain on the Malcev completion of  $\mathcal{I}$.
Here, we  prove that the abelian group 
$\Gamma_2 \mathcal{I} / \Gamma_3 \mathcal{I}$ 
is torsion-free, 
and we describe it as a lattice 
in a rational vector space.
As an application,
the group $\mathcal{I}/\Gamma_3 \mathcal{I}$
is computed, and it is shown to embed 
in the group of homology cylinders 
modulo the surgery relation of $Y_3$-equivalence.
\end{abstract}

\maketitle

\vspace{-0.5cm}
\tableofcontents
\vspace{-1cm}

\section{Introduction} \label{sec:intro}

Let $\Sigma$ be a compact, connected, oriented surface with boundary.
Its \emph{mapping class group} $\calM:=\calM(\Sigma)$ 
is the group of isotopy classes of orientation-preserving 
homeomorphisms of $\Sigma$ that restrict to the identity on  $\partial \Sigma$.
Here, for simplicity,
we assume that $\Sigma$ has exactly one boundary component 
and genus $g\geq 3$.
The subgroup $\calI:=\calI(\Sigma)$ of $\calM$
that acts trivially on the homology $H:= H_1(\Sigma;\Z)$ of the surface
is known as \emph{the Torelli group} of $\Sigma$.
It contains the  subgroup $\calK:=\calK(\Sigma)$ 
that is generated by Dehn twists along bounding simple closed curves,
the latter being known as the \emph{Johnson subgroup} of~$\Sigma$.

Johnson fully determined in \cite{Jo85b}
the abelianization $\calI_{\ab}:=\calI/[\calI,\calI]$ of the Torelli group. 
In particular, he proved that
 $\calI_{\ab}$ has non-trivial torsion
(with only order-two elements),
and that its torsion-free abelianization
$
\calI_{\abf} 
:= \calI_{\ab} \big/ \hbox{Tors}(\calI_{\ab} )
$
is canonically isomorphic to $\Lambda^3 H$
through the \emph{first Johnson homomorphism}\\[-0.5cm]
$$
\tau_1: \calI \longrightarrow \Lambda^3 H,
$$
which encodes the action of $\calI$
on the second nilpotent quotient of $\pi_1(\Sigma)$.
Since the kernel of~$\tau_1$ coincides with the Johnson subgroup,
we have $\calI_{\abf}=\calI/\calK$
or, equivalently, 
$\hbox{Tors}(\calI_{\ab} ) = \calK/[\calI,\calI]$.

Consider now the lower central series 
$\Gamma_1\calI \supset \Gamma_2\calI \supset \Gamma_3\calI \supset \cdots $
of the Torelli group, which is defined by $\Gamma_1\calI=\calI$,
$\Gamma_2 \calI =[\calI,\calI]$,
$\Gamma_{3} \calI = [\calI,[\calI,\calI]]$ and so on.
Its associated graded
$$
\operatorname{Gr}^\Gamma \calI = \bigoplus_{k\geq 1}
 \frac{\Gamma_k \calI}{\Gamma_{k+1} \calI}
$$
 is a graded Lie $\Z$-algebra,
which is generated by its  degree $1$ part, namely by $\calI_{\ab}$.
As a general fact,
the associated graded of the lower central series of a group
is, with rational coefficients, the same as the associated graded
of the natural filtration on the Malcev Lie algebra of the group.
In the case of the Torelli group, much more is known:
in the fundamental work \cite{Ha97}, Hain proved that
the degree-completion of
$\big(\operatorname{Gr}^\Gamma \calI\big)\otimes \Q$ 
is (non-canonically) isomorphic 
to the Malcev Lie algebra of $\calI$
and, furthermore, he identified the latter graded Lie $\Q$-algebra 
 in the following way.
 
Let $H^\Q:=H\otimes \Q$ and consider the free Lie $\Q$-algebra $\Lie(\Lambda^3 H^\Q)$
generated by $\Lambda^3 H^\Q$.
Clearly, the map $J^\Q$ defined by the composition
\begin{equation} \label{eq:JQ}
\xymatrix{
\Lie(\Lambda^3 H^\Q)
\ar[rr]^-{\Lie(\tau_1^{-1})}_-\simeq
\ar@/_0.7cm/@{-->}[rrrr]_{J^\Q}
& & \Lie(\calI_{\ab}\otimes \Q) \ar@{->>}[rr]
&& \big(\operatorname{Gr}^\Gamma \calI\big)\otimes \Q
}    
\end{equation}
is surjective, 
so that the determination of 
$\big(\operatorname{Gr}^\Gamma \calI\big)\otimes \Q$
resumes at the computation of the ideal of \emph{rational relations} $R^\Q:=\ker(J^\Q)$.
Note that the conjugacy action of $\calM$ on $\calI$
induces an action of the symplectic group
$\Sp(H^\Q)$ on 
$\big(\operatorname{Gr}^\Gamma \calI\big)\otimes \Q$,
such that $J^\Q$ is $\Sp(H^\Q)$-equivariant.
Here the symplectic form $\omega$ on $H^\Q$ is given by 
the homology intersection pairing.
Then, it turns out that the ideal $R^\Q$ is generated 
by its degree $2$ part $R^\Q_2$ 
for $g\geq 4$ \cite{Ha97,Ha15}. Furthermore, for $g\geq 3$,
the subspace $R_2^\Q$ can be described 
as the $\Sp(H^\Q)$-submodule of 
$\Lie_2(\Lambda^3 H^\Q) \simeq \Lambda^2(\Lambda^3 H^\Q)$
generated by two explicit elements   \cite{Ha97,HS00}.
Thus, using Morita's works on the Casson invariant 
and the second Johnson homomorphism \cite{Mor89,Mor91},
one obtains an explicit description of 
$$
(\Gamma_2 \calI / \Gamma_3 \calI)\otimes \Q \simeq
\Lambda^2(\Lambda^3 H^\Q) / R_2^\Q.
$$
Besides, the subspace $R_2^\Q$ 
 coincides with the kernel of a certain map 
$B$ which is defined explicitly 
on $\Lambda^2(\Lambda^3 H^\Q)$
in terms of contractions with $\omega$: 
see equation \eqref{eq:B_map} below 
for the exact definition.
In the sequel,
we shall refer to $B$ as the \emph{diagrammatic bracket}
since it corresponds in \cite{HM09}
to a restriction of the Lie bracket
of so-called ``symplectic Jacobi diagrams'',
and it gives rise there to a diagrammatic description 
of $(\Gamma_2 \calI / \Gamma_3 \calI)\otimes \Q $.\\

In this paper, we are interested in the $\Z$-module structure
of  $\operatorname{Gr}^\Gamma \calI$, 
which (beyond its rationalization)
seems to be largely unknown in degree greater than $1$. 
A first step in this direction is the following.\sdl
\textbf{Theorem~A.}
\emph{Let $g\geq 3$.
We have $[\calI,[\calI, \calI]] =[\calI,\calK]$.}\sdl
This is proved in Section~\ref{sec:proof_Th_A} 
as follows.
Firstly, by using Johnson's descriptions \cite{Jo85b} of 
the group $\calI_{\ab}$ and its subgroup
$\calK/[\calI,\calI]$,
we observe that the triviality of the quotient
$[\calI,\calK]/[\calI,[\calI, \calI]]$ follows from 
the nullity of two special elements
(Lemma \ref{lem:c1_c2}). 
This nullity is then ensured  by using some appropriate  ``Push'' maps 
for surfaces related to $\Sigma$.

Theorem~A has the following two  consequences.
On the one hand, the torsion of 
$\Gr^\Gamma _1\calI =\calI_{\ab}$
does not ``propagate'' to higher-degree torsion elements 
in $\Gr^\Gamma \calI$.
(However, note  that  $\operatorname{Gr}^\Gamma \calI$ 
does contain  non-trivial torsion 
in any odd degree: 
see the forthcoming paper \cite{NSS24}, 
which builds upon results of  \cite{NSS22} and \cite{KRW23}.)
On the other hand,  
the map \eqref{eq:JQ}  also exists with integral coefficients
in degree greater than $1$:
\begin{equation} \label{eq:J}
\xymatrix{
\Lie_{\geq 2}(\Lambda^3 H)
\ar[rr]^-{\Lie_{\geq 2}(\tau_1^{-1})}_-\simeq
\ar@/_0.7cm/@{-->}[rrrr]_{J}
& & \Lie_{\geq 2}(\calI_{\abf}) \ar@{->>}[rr]
&& \operatorname{Gr}^\Gamma_{\geq 2} \calI
}    
\end{equation}

Consequently, 
we can consider the ideal of \emph{integral relations} $R:=\ker(J)$, where $J$ is the above map.
Our next result, which is proved in Section~\ref{sec:proof_Th_B},
computes explicitly  its degree $2$ part $R_2$.\sdl
\textbf{Theorem~B.}
\emph{Let $g\geq 3$.
We have $R_2=\ker(B)\cap \Lambda^2 (\Lambda^3 H)$.}\sdl
Of course, the inclusion
$R_2 \subset \ker(B)\cap \Lambda^2 (\Lambda^3 H)$ 
follows from the above-mentioned identity 
$R_2^\Q=\ker(B)$ \cite{HM09}:
alternatively, this inclusion is  a direct consequence of 
the relationship between the homomorphism~$B$ and the pair 
(second Johnson homomorphism, Casson invariant)
on $[\calI,\calI]$,
which will be specified in Proposition \ref{prop:invariants}.
The  proof of the converse inclusion 
(i.e., the fact that any element of $\Lambda^2 (\Lambda^3 H)$
with trivial diagrammatic bracket
is a quadratic relation 
in the graded Lie ring $\Gr^\Gamma \calI$)
constitutes the most technical part of this paper.
In contrast with the above-mentioned characterizations 
of $R_2^\Q$
which use the representation theory of  $\Sp(H^\Q)$,
our proof of Theorem~B is rather pedestrian 
and proceeds in two steps. First,  
we fix a symplectic basis $S$ of~$H$,
and we produce an explicit finite generating system of 
$\ker(B)\cap \Lambda^2 (\Lambda^3 H)$
under the action of the subgroup
of the integral symplectic group $\Sp(H)$ that 
stabilizes $S\cup (-S)$:
see Theorem~\ref{th:K}. 
Second, we check that each of these generators 
defines a quadratic relation for $\Gr^\Gamma \calI$:
one of the ingredients here 
is the ``topological IHX relation''
that has been identified by Gervais and Habegger in \cite{GH02}.

Being isomorphic to a  lattice in the $\Q$-vector space $\Lambda^2(\Lambda^3 H^\Q)/ \ker(B)$,
the group  $\Gamma_2 \calI/\Gamma_3 \calI$ turns out to be torsion-free.
Another output of (the proof of)  Theorem~B  is an explicit description
of the quadratic relations.
Although it is known to be true with \emph{rational} coefficients 
from the above-mentioned work of Hain \cite{Ha97},
the graded Lie $\Z$-algebra $\operatorname{Gr}^\Gamma \calI$ 
may not be quadratically presented (even for $g$ large enough).
In this hypothesis,
it would still be necessary 
to determine the relations in higher degrees.

In Section \ref{sec:groups}, 
we  derive from Theorem~A and Theorem~B other consequences.
(Actually, the reader  only interested in these applications of Theorem~A and Theorem~B,
may  skip the detailed proofs of the latter, which are contained in
 \S \ref{subsec:nullity}-\S \ref{subsec:Push}
and \S \ref{subsec:strategy}-\S \ref{subsec:vanishing} respectively,
and jump directly to  Section \ref{sec:groups}.)
The exact structure of the groups $\Gamma_2 \calI/\Gamma_3 \calI$, $\calK/\Gamma_3 \calI$ 
and  $\calI/\Gamma_3 \calI$ is obtained in Theorem~\ref{th:groups} and Theorem~\ref{th:I/Gamma3I}.
Note that we do not consider here
the case of a closed surface.  
Nonetheless, it is very likely that one could derive 
from our study  similar results
for the same subquotients of the Torelli group 
of the surface $\clo{\Sigma}$ 
that is obtained from $\Sigma$
by gluing a closed disk along $\partial \Sigma$. We leave the study of the closed case, which is more technical, to future work.

Finally, there are two fundamental problems
about the lower central series of $\calI$ 
in connection with other filtrations,
to which we can provide (in low degrees) a few new pieces of answer.
On the one hand, recall that the ``Johnson filtration''
$\calM=\calM[0] \supset \calM[1] \supset \calM[2] \supset \cdots$
 consists of the subgroups of the mapping class group
acting trivially  
on the successive nilpotent quotients of $\pi_1(\Sigma)$.
On the Torelli group $\calM[1]=\calI$, this filtration is larger 
than the lower central series, and the following questions arise.\sdl
\textbf{Problem C.} 
\emph{Determine the kernel and the image of the homomorphism
\begin{equation}  \label{eq:graded}
\bigoplus_{k\geq 1} \frac{\Gamma_k \calI}{\Gamma_{k+1} \calI}
\longrightarrow 
\bigoplus_{k\geq 1} \frac{\calM[k]}{\calM[k+1]}
\end{equation}
from the associated graded of the lower central series of $\calI$
to the associated graded  of the Johnson filtration of $\calM$.}\sdl
Of course, the map \eqref{eq:graded} 
is very well-understood in degree $k=1$ 
thanks to~\cite{Jo85b}.
Subsequent works \cite{Mor89,Yok02,Faes23} also give
a nice description of its image in degree $k=2$.
Besides, it follows from  Hain's results \cite{Ha97}
that \eqref{eq:graded} is \emph{rationally} surjective;
see also  \cite{MSS20,KRW23}
for  recent advances in the determination of the kernel of
\eqref{eq:graded} with \emph{rational} coefficients.
Having ruled out the possibility of torsion in 
$\Gamma_2 \calI/\Gamma_3 \calI$, 
we determine the kernel of \eqref{eq:graded} for $k=2$
(see Corollary \ref{cor:intersection}).
Moreover, using the main result of \cite{Faes22},
we obtain that \eqref{eq:graded} is surjective for $k=3$
(see Corollary \ref{cor:M[3]}).

On the other hand, let $\mathcal{IC}:= \mathcal{IC}(\Sigma)$
be the monoid of homology cylinders over $\Sigma$.
The study of $\mathcal{IC}$ by means of surgery techniques 
and finite-type invariants has been  initiated 
by Goussarov and Habiro in \cite{Gou00,Ha00}. 
Following this study, 
we are  interested in the comparison between 
the lower central series of $\calI$
and the ``$Y$-filtration''  
$\mathcal{IC} = Y_1 \mathcal{IC} \supset  Y_2 \mathcal{IC} 
\supset  Y_3 \mathcal{IC} \supset \cdots $ 
Indeed,  $\calI$ embeds into $\mathcal{IC}$ via the ``mapping cylinder'' construction, which preserves those filtrations.\sdl
\textbf{Problem D.} 
\emph{Determine the kernel and the image of the homomorphism
\begin{equation}  \label{eq:graded_Y}
\bigoplus_{k\geq 1} \frac{\Gamma_k \calI}{\Gamma_{k+1} \calI}
\longrightarrow 
\bigoplus_{k\geq 1} \frac{Y_k\mathcal{IC}}{Y_{k+1}}.
\end{equation}
from the associated graded of the lower central series of $\calI$
to the associated graded  of the $Y$-filtration of $\mathcal{IC}$.}\sdl
The case $k=1$ is already known \cite{Ha00,MM03}.
Besides, the question of the injectivity for arbitrary~$k$, 
which has already been asked in \cite{HM09},
recently received a partial answer in \cite{KRW23}
for \emph{rational} coefficients
(in the stable range of the genus $g$,
and up to central elements).
Using some of the results of~\cite{MM13},
we solve Problem D in degree $k=2$  (see Theorem~\ref{th:IC/Y3}).\\

\noindent
\textbf{Acknowledgment.}
The first author was supported by Grant 200020\textunderscore 207374 of the Swiss National Science Foundation.
The work of the second author was partly funded by the project “AlMaRe” (ANR-19-CE40-0001-01);
the IMB receives support from the EIPHI Graduate School (ANR-17-EURE-0002).
The work of the third author was partly supported by JSPS KAKENHI Grant Numbers JP18K03310 and
JP22K03298.
\hfill $\blacksquare$ \\

\noindent
\textbf{Conventions.}
Given a group $G$ and a normal subgroup $N$ of $G$,
the class modulo $N$ of an element $x\in G$ is denoted by $\{x\}_N\in G/N$,
or simply  $\{x\}$ (or even just $x$) if there is no risk of confusion.
For any elements $x,y\in G$, 
we set ${}^xy=xyx^{-1}$, $y^x=x^{-1}yx$
and  $[x,y]=xyx^{-1}y^{-1}$.

If not specified, the ground ring for linear algebra is $\Z$.
(So, for instance, a ``module'' means a $\Z$-module,
i.e., an abelian group.)
For  $k\geq 1$, $\Z_k := \Z/k\Z$ 
is
the cyclic group of order $k$. \hfill $\blacksquare$

\section{Proof of Theorem~A}

\label{sec:proof_Th_A}

To prove Theorem~A, we need to show that the abelian group
$$
Q:= \frac{[\calI,\calK]}{[\calI,[\calI, \calI]]} 
$$
is trivial, and for this, we consider the bilinear map
$$
\Upsilon: \calI_{\ab} \times \frac{\calK}{[\calI,\calI]}
\longrightarrow Q
$$
defined by restricting the Lie bracket of $\Gr^\Gamma \calI$
to degrees $1+1$.
Let  $(\gamma_+,\gamma_-)$ be a pair of simple closed curves
cobounding a subsurface of genus $1$ in $\Sigma$,
and recall from \cite{Jo79} that $\calI$ is normally generated in $\calM$ 
by  the product $p:=T_{\gamma_+} T_{\gamma_-}^{-1}$ of opposite Dehn twists: 
hence $ \calI_{\ab}$ is generated by $\{p\}_{[\calI,\calI]}$ as an $\Sp(H)$-module.
Since $\Upsilon$ is surjective and $\Sp(H)$-equivariant, 
the triviality of $Q$ is equivalent to the nullity
of the map 
$$
\upsilon:=\Upsilon(\{p\},-): {\calK}/{[\calI,\calI]} \longrightarrow Q,
\ \{k\}_{[\calI,\calI]} \longmapsto \big\{[p,k]\big\}_{[\calI,[\calI, \calI]]}.
$$
By Johnson's work
on the abelianized Torelli group~\cite{Jo85b},
the sources of the map $\Upsilon$ and $\upsilon$ are perfectly well understood.
We now review his work very briefly.

\subsection{The abelianization of the Torelli group}

Recall that, by our assumption, the genus $g$ of $\Sigma$ is at least $3$.
Then, the main result of  \cite{Jo85b} is an 
$\Sp(H)$-equivariant isomorphism 
\begin{equation} \label{eq:ab_Torelli}
\xymatrix{
\calI_{\ab} \ar[r]^-{(\tau_1,\beta)}_-\simeq & 
\Lambda^3 H \times_{\Lambda^3 H \otimes \Z_2} B_{\leq 3}(\calQ)
}    
\end{equation}
where $\calQ$ denotes the $\Z_2$-affine space of quadratic forms
with polar form $\omega \otimes \Z_2$, the space
$B_{\leq k}(\calQ)$ consists  of polynomial functions $\calQ\to \Z_2$
of degree at most $k$, the map
$\beta:\calI \to B_{\leq 3}(\calQ)$
is the Birman--Craggs homomorphism,
and the maps defining the above fibered product are as follows:
the map $\Lambda^3 H \to \Lambda^3 H \otimes \Z_2$
is the canonical homomorphism,
and the map $B_{\leq 3}(\calQ) \to \Lambda^3 H \otimes \Z_2$
takes the third differential of cubic functions.

Recall that both $\tau_1$ and $\beta$ can be  computed as follows
on the product  $T_{\varepsilon}T_\delta^{-1}$ of opposite Dehn twists 
along a pair $(\varepsilon,\delta)$
of disjoint and cobounding, simple closed curves:
\begin{equation} \label{eq:tau1_beta}
\tau_1\big(T_{\varepsilon}T_\delta^{-1}\big)= 
-\sum_{i=1}^h u_i \wedge v_i \wedge e
\quad \hbox{and} \quad
\beta\big(T_{\varepsilon}T_\delta^{-1}\big) = 
\sum_{i=1}^h \overline{u_i} \, \overline{v_i} \, (\overline{e}+ \overline{1}).
\end{equation}
Here $(u_1,\dots,u_h,v_1,\dots,v_h)$ is a symplectic basis of
the subsurface cobounded by $\varepsilon$ and  $\delta$
(with the orientation induced from $\Sigma$),
and $e:=\{\varepsilon\}$ is the homology class of 
$\varepsilon$ and $\delta$
(which have been oriented so that the oriented boundary 
of the previous subsurface is $\varepsilon \cup (-\delta)$).
Furthermore, we use the following notations:
$\overline{1}: \calQ \to \Z_2$ is the constant non-trivial function
and, for any $h\in H$, $\overline{h}: \calQ \to \Z_2$ 
is the evaluation at $h$.

\begin{remark} \label{rem:Sp-action}
Note that
\begin{equation} \label{eq:quadriticity}
    \overline{h_1+h_2} = \overline{h_1} + \overline{h_2}
    + \omega(h_1,h_2)\, \overline{1},
    \quad \hbox{for any $h_1,h_2\in H$}.
\end{equation}    
Besides, 
the canonical action of $\Sp(H)$ on $B_{\leq k}(\calQ)$ 
preserves the algebra structure, 
and is such that
$M\cdot \overline{h} = \overline{M(h)}$ 
for any $M\in \Sp(H)$ and $h\in H$. \hfill $\blacksquare$
\end{remark}

As a consequence of \eqref{eq:ab_Torelli}, we also have an 
$\Sp(H)$-equivariant isomorphism
\begin{equation} \label{eq:K_quotient}
\xymatrix{
 \calK /[\calI,\calI] \ar[r]^-\beta_-\simeq & B_{\leq 2}(\calQ),
 }
\end{equation}
which can be computed as follows:
for any bounding simple closed curve $\delta$, we have
\begin{equation} \label{eq:beta_BSCC} 
\beta(T_\delta) = 
\sum_{i=1}^h \overline{u_i} \cdot \overline{v_i}
\end{equation}
where $(u_1,\dots,u_h,v_1,\dots,v_h)$ is a symplectic basis of
the subsurface of $\Sigma$ that is bounded by $\delta$.

\subsection{Nullity of the map $\upsilon$}
\label{subsec:nullity}

\begin{figure}[h]
\includegraphics[scale=0.8]{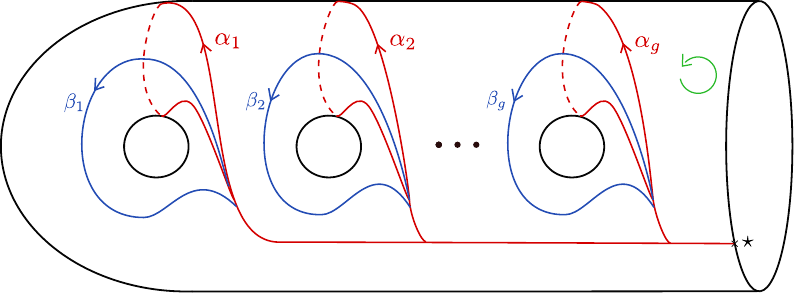}
\caption{A system of ``meridians  \& parallels'' in  $\Sigma$}
\label{fig:basis}
\end{figure}

Let $(\alpha_1,\dots,\alpha_g,\beta_1,\dots,\beta_g)$
be a system of ``meridians  \& parallels'' in $\Sigma$
as shown in Figure \ref{fig:basis}.
These are oriented simple closed curves
and, with the arcs to the basepoint shown on the same figure,
they define a basis of the free group
$\pi:=\pi_1(\Sigma, \star)$ with $\star \in \partial \Sigma$.
The corresponding basis of $H=H_1(\Sigma;\Z)$ is denoted by 
\begin{equation} \label{eq:symplectic_basis}
 \{a_1,\dots,a_g,b_1,\dots,b_g\}.
\end{equation}
In the sequel, we take $\gamma_+$ and $\gamma_-$
to be the curves shown
in  Figure~\ref{fig:gamma_c_1_c_2}.

\begin{figure}[h]
\includegraphics[scale=0.55]{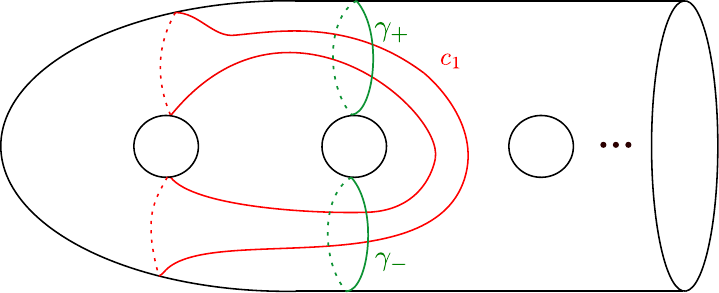} \quad \includegraphics[scale=0.55]{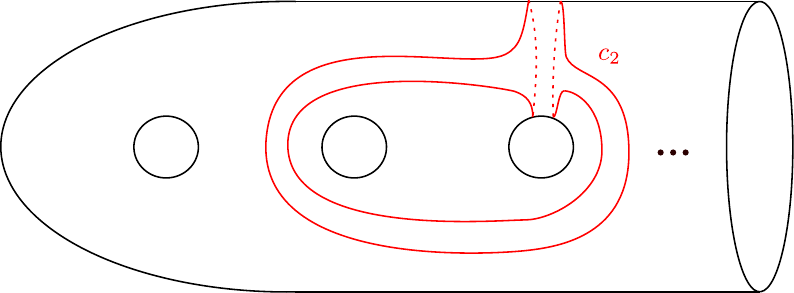} 
\caption{The cobounding pair $(\gamma_+,\gamma_-)$,
and the bounding simple closed curves $c_1,c_2$}
\label{fig:gamma_c_1_c_2}
\end{figure}

Let $\calS_p$ be the subgroup of $\Sp(H)$ stabilizing  the class
$\{p\} \in \calI_{\ab}$ of  $p=T_{\gamma_+} T_{\gamma_-}^{-1}$.

\begin{lemma} \label{lem:S_p}
As an $\calS_p$-module, $B_{\leq 2}(\calQ)$ is generated by the following elements:
\begin{equation} \label{eq:generators}
\overline{a_1}\, \overline{b_1}, \ \overline{a_2}\, \overline{b_2}, \
\overline{a_3}\, \overline{b_3}, \ \overline{a_1}\, \overline{b_2}, \
\overline{a_3}\, \overline{b_2}.
\end{equation}
\end{lemma}

\begin{proof}
It follows from the isomorphism \eqref{eq:ab_Torelli} and from
formulas~\eqref{eq:tau1_beta}
that  $\calS_p$ is the subgroup of $\Sp(H)$ fixing both
$a_1 \wedge b_1 \wedge a_2 \in \Lambda^3 H$ and 
$\overline{a_1}\, \overline{b_1} 
\, (\overline{1} + \overline{a_2}) \in  B_{\leq 3}(\calQ)$. 
Thus, using Remark \ref{rem:Sp-action},
we easily check that the following elements of $\Sp(H)$ belong to $\calS_p$:
\begin{itemize}
\item $C_1$ maps $b_1$ to $b_1+a_1$
and fixes all other elements of \eqref{eq:symplectic_basis};
\item for $i\neq g$, $D_i$ maps $(b_i,b_{i+1})$ to $(b_i+a_{i+1},b_{i+1}+a_i)$
and fixes all other elements of \eqref{eq:symplectic_basis};
\item for $i\neq 2$, $E_i$ maps $(a_i,b_i)$ to $(-b_i,a_i)$ 
and fixes all other elements of \eqref{eq:symplectic_basis};
\item for $i,j\geq 3$, $F_{ij}$ exchanges $(a_i,b_i)$ with $(a_j,b_j)$
and fixes all other elements of \eqref{eq:symplectic_basis}.
\end{itemize}
Let $V$ be the $\calS_p$-submodule of $B_{\leq 2}(\calQ)$ generated by \eqref{eq:generators}.
As a $\Z_2$-vector space, $B_{\leq 2}(\calQ)$ 
is generated by the following elements:
\begin{center}
\begin{tabular}{llll}
    (i) \ \ $\overline{a_i}\,\overline{b_i}$ &
    (ii) \ $\overline{a_i}\,\overline{b_j}$ ($i\neq j$) \qquad &
    (iii) \ $\overline{a_i}\,\overline{a_j}$ ($i< j$) \qquad &  \\
    (iv)  \  $\overline{b_i}\,\overline{b_j}$ ($i<j$) \qquad &
    (v) \ $\overline{a_i}$ & (vi) \ $\overline{b_i}$ & (vii) \ $\overline{1}$
\end{tabular}
\end{center}
Hence, it suffices to check that the above elements belong to $V$:
\begin{itemize}
\item
For $i\geq 4$, we have $\overline{a_i}\overline{b_i}\in V$
since  $\overline{a_i}\overline{b_i} 
= F_{3i} \cdot \overline{a_3}\overline{b_3}$, which gives the proof for (i).
\item We have $D_2 \cdot \overline{a_2} \overline{b_2} 
= \overline{a_2} \overline{b_2} + \overline{a_2} \, \overline{a_3}$,
hence $\overline{a_2} \,\overline{a_3}\in V$
which implies $\overline{a_2} \,\overline{a_i} =
F_{3i}\cdot \overline{a_2}\, \overline{a_3}  \in V$ for all $i\geq 3$. 
Next, $D_1 \cdot \overline{a_1} \overline{b_1}
= \overline{a_1} \overline{b_1} + \overline{a_1}\, \overline{a_2}$
hence $\overline{a_1}\, \overline{a_2}\in V$. 
Moreover, $D_2\cdot \overline{a_1} \overline{b_2}
=\overline{a_1} \overline{b_2}+ \overline{a_1}\, \overline{a_3}$
so that $\overline{a_1} \,\overline{a_3}\in V$
which implies $\overline{a_1}\, \overline{a_i}
=F_{3i}\cdot \overline{a_1}\, \overline{a_3} \in V$ for all $i\geq 3$. 
Finally, $D_3 \cdot \overline{a_3} \overline{b_3} = 
\overline{a_3} \overline{b_3}  + \overline{a_3}\, \overline{a_4}$
 so that $\overline{a_3}\, \overline{a_4} \in V$,
which implies $\overline{a_i}\, \overline{a_j}  \in V$ 
for all $3 \leq i <j$. This proves that elements of type (iii) belong to $V$.
\item For all $i\geq 3$, we have 
$\overline{a_i} \overline{b_2} = R_{3i}\cdot \overline{a_3} \overline{b_2}
\in V$. Moreover, 
for all $j \neq 2$ and $i\neq j$, 
we deduce from (iii) that $\overline{a_i} \overline{b_j} 
= E_j \cdot  \overline{a_i} \overline{a_j} \in V$.
Thus we have proved that elements of type (ii) belong to $V$.
\item For all $i \neq 2$ and $j\neq i$, 
we deduce from (ii) that $\overline{b_i} \overline{b_j} 
= E_i \cdot  \overline{a_i} \overline{b_j} \in V$, which proves (iv).
\item We have $D_1\cdot \overline{a_1} \overline{b_2}=
\overline{a_1} \overline{b_2} + \overline{a_1}$
which implies that $\overline{a_1}\in V$.
We have $D_1 \cdot \overline{b_1} \overline{a_2}
= \overline{b_1} \overline{a_2} + \overline{a_2}$
and so, by (ii), we get $\overline{a_2}\in V$.
Next, we have $C_1 \cdot \overline{b_1} \overline{a_3} =
\overline{a_1}\, \overline{a_3} + \overline{b_1} \overline{a_3}  
+ \overline{a_3}$ which, by (ii) and (iii), implies that
$\overline{a_3} \in V$. It follows that 
$\overline{a_i} = F_{3i} \cdot \overline{a_3}\in V$ for all $i\geq 3$,
which concludes for (v).
\item We have $C_1 \cdot \overline{b_1}\, \overline{b_2}=
\overline{b_1}\, \overline{b_2} +  \overline{a_1} \overline{b_2} 
+ \overline{b_2}$ so, using (ii) and (iv), 
we obtain $\overline{b_2} \in V$.
Since $\overline{b_i}=E_i \cdot \overline{a_i}$ for every $i\neq 2$,
we then deduce case (vi) from case (v).
\item Finally, $C_1 \cdot \overline{b_1} = \overline{b_1} 
+ \overline{a_1} + \overline{1}$, and case (vii) follows too.
\end{itemize}
\end{proof}

\begin{lemma} \label{lem:c1_c2}
As an $\calS_p$-module, 
the image of $\upsilon: \calK/[\calI,\calI]\to Q$ 
is generated by $\upsilon(T_{c_1})$
and $\upsilon(T_{c_2})$, 
where $c_1$ and $c_2$ are 
the simple closed curves shown in Figure \ref{fig:gamma_c_1_c_2}.
\end{lemma}

\begin{proof}
Since  the isomorphism  \eqref{eq:K_quotient} is $\Sp(H)$-equivariant,
we deduce from Lemma \ref{lem:S_p}  that the $\calS_p$-module
$\calK/[\calI,\calI]$ is generated by the classes of any 
$f_1,f_2,\dots, f_5 \in \calK$ such that
$$
\beta(f_1)=\overline{a_1}\cdot \overline{b_1}, \ 
\beta(f_2) = \overline{a_2}\cdot \overline{b_2}, \
\beta(f_3) = \overline{a_3}\cdot \overline{b_3}, \ 
\beta(f_4) = \overline{a_1}\cdot \overline{b_2}, \
\beta(f_5) = \overline{a_3}\cdot \overline{b_2}.
$$
Using \eqref{eq:beta_BSCC} and considering 
the following bounding curves $\gamma_1,\gamma_2,\gamma_3$, 
we see that 
$f_1:=T_{\gamma_1}$,  $f_2:=T_{\gamma_1} T_{\gamma_2}$ 
and $f_3:= T_{\gamma_2} T_{\gamma_3}$ suit our purposes:
$$
\includegraphics[scale=0.5]{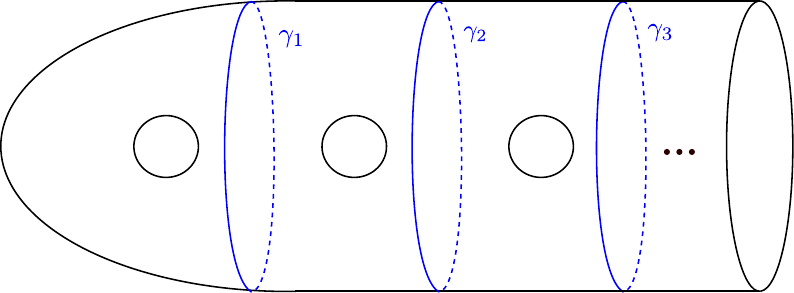}
$$
Besides, \eqref{eq:beta_BSCC}  also gives
$$
\beta(T_{c_1}) = 
\overline{a_1} \cdot \overline{b_1+b_2} \by{eq:quadriticity}
\overline{a_1} \cdot \overline{b_1} + \overline{a_1} \cdot \overline{b_2}
\quad \hbox{and} \quad 
\beta(T_{c_2}) = 
\overline{a_3} \cdot \overline{b_2+b_3}  \by{eq:quadriticity}
\overline{a_3} \cdot \overline{b_2} 
+ \overline{a_3} \cdot \overline{b_3},
$$
showing that we can take $f_4 := f_1 T_{c_1}$
and $f_5 :=  f_3 T_{c_2}$.

Since the map $\upsilon$ is $\calS_p$-equivariant and since its source
is $\calS_p$-generated by $\{f_1\},\dots,\{f_5\}$,
its image is $\calS_p$-generated by $\upsilon(\{f_1\}),\dots,\upsilon(\{f_5\})$.
Since  the curve $\gamma_i$ is disjoint from $\gamma_+$ and $\gamma_-$
for each $i\in \{1,2,3\}$, we obviously have $\upsilon(\{T_{\gamma_i}\})=0$.
Hence, by definition of $f_1,\dots,f_5$, 
the image of $\upsilon$ is $\calS_p$-generated by $\upsilon(\{T_{c_1}\})$
and $\upsilon(\{T_{c_2}\})$.
\end{proof}

\emph{Assume} now  that both $[p,T_{c_1}]$ and $[p,T_{c_2}]$ belong  
to $\Gamma_3 \calI$:
this will follow from Example \ref{ex:c1}
and Example \ref{ex:c2}, respectively, in the next subsection.
Then we have $\upsilon(T_{c_1})=\upsilon(T_{c_2})=0$, 
and Lemma~\ref{lem:c1_c2} implies that the map $\upsilon$ is zero.
Therefore, the group $Q={[\calI,\calK]}/{[\calI,[\calI, \calI]]}$ is trivial,
which proves Theorem~A.

\subsection{The ``Push'' map and its applications}
\label{subsec:Push}

To prove that $[p,T_{c_i}] \in \Gamma_3 \calI$ for $i=1,2$
(see Examples \ref{ex:c1} and  \ref{ex:c2} below),
we  use the ``Push'' map. So we start by reviewing this construction.

Let $W$ be a compact, connected, oriented surface 
with $b\in\{1,2\}$ boundary components, 
and fix a connected component $\delta$ of $\partial W$.
The surface obtained from~$W$ by gluing a closed disk $D$ along $\delta$
is denoted by $\clo{W}$.
The unit tangent bundle of $\clo{W}$ is denoted by $UT(\clo{W})$;
we choose a base point $\overrightarrow{\ast} \in UT(\clo{W})$ 
which projects to the center $\ast \in D$.

There is a homomorphism
$\calM(W)\to \calM(\clo{W})$ that maps any $f$ to $\clo{f}:= f\cup \id_D $.
Moreover, there is a anti-homomorphism
$\Push:\pi_1\big(UT(\clo{W}),\overrightarrow{\ast}\big) \to \calM(W)$
given by
$
\Push(\hbox{fiber}) = T_{\delta},
$
and by
\begin{equation} \label{eq:Push}
\Push(\overrightarrow{\gamma}) = T_{\gamma_+} T_{\gamma_-}^{-1}    
\end{equation}
for any (smooth) oriented simple closed curve $\gamma$ in $\clo{W}$, 
whose unit tangent vector field
$\overrightarrow{\gamma}$ passes through $\overrightarrow{\ast}$:
here $\gamma_+$ (resp$.$ $\gamma_-$) is the boundary component
of a regular neighborhood $\hbox{N}(\gamma)$ of~$\gamma$   (containing $D$)
where the orientation induced by $\hbox{N}(\gamma)$ is the same as
(resp. opposite to) the orientation parallel to $\gamma$.
Then, the  \emph{Birman's exact sequence} is
$$
\xymatrix{
\pi_1\big(UT(\clo{W}),\overrightarrow{\ast}\big)\ar[r]^-{\Push} & \calM(W)
 \ar[r]^-{\clo{(\,\cdot\,)}} & \calM(\clo{W}) \ar[r] & 1,
}
$$
see \cite[\S 4.2.5]{FM12} or \cite{Put09}, for instance.
Furthermore,  the ``Push'' map is injective if $\chi(\clo{W})<0$,
but we shall not need this fact.

Let $V$ be another (compact, connected, oriented) surface
with $c\in\{1,2\}$ boundary components, 
and assume that one connected component of $\partial V$
is identified to $\delta$.
Then, by gluing $W$ and $V$ along $\delta$, we get a new surface $W^+=W \cup V$,
which has $b+c-2$ boundary components.

There is a homomorphism
$\calM(W)\to \calM({W}^+)$ that maps any $f$
to ${f}^+:= f\cup \id_V$.
In the sequel, we consider the following composition
$$
\xymatrix{
\pi_1\big(UT(\clo{W}),\overrightarrow{\ast}\big)\ar[r]^-{\Push}
\ar@/_0.5cm/@{-->}[rr]_{\Push^+} & \calM(W)
 \ar[r]^-{(\,\cdot\,)^+} & \calM(W^+) 
}
$$

\begin{proposition} \label{prop:push-push}
Assume that the curve $\delta$ bounds in $W^+$
(i.e. $b=1$ or $c=1$), 
so that $W^+$ has at most one boundary component.
Then, we have the following:
\begin{enumerate}
\item  $\Push^+$ takes values in $\calI(W^+)$;
\item  for any $u\in \pi_1\big(UT(\clo{W}),\overrightarrow{\ast}\big)$
and  $w\in \calM(W)$  such that $\clo{w}$ acts trivially on the quotient
$\pi_1(\clo{W},\ast)/\Gamma_{k} \pi_1(\clo{W},\ast)$ for some $k\geq 2$,
the commutator $\big[w^+,\Push^+(u)\big]$ in $\calM(W^+)$ belongs to $\Gamma_{k} \calI(W^+)$.
\end{enumerate}
\end{proposition}

\begin{proof}
(1) By assumption, 
$\Push^+(\operatorname{fiber})=T_\delta \in \calM(W^+)$ 
is the Dehn twist along a bounding simple closed curve and, so,
belongs to $\calI(W^+$).
Thus, it remains to check that, for any
simple closed curve $\gamma \subset \clo{W}$  as in \eqref{eq:Push}, 
we have $\Push(\overrightarrow{\gamma})\in \calI(W^+)$. 
Indeed, if $c=1$,  then $\gamma_+$ and $\gamma_-$ 
cobound a subsurface in $W^+$ (which contains $V$).
Besides, if $b=1$, then (depending on whether $\gamma$ is 
separating or not in $\clo{W}$) 
either each of $\gamma_+$ and $\gamma_-$  bounds a subsurface in $W \subset W^+$, or,
$\gamma_+$ and $\gamma_-$  cobound a subsurface in $W\subset W^+$.\\
(2) For $k=2$, the statement follows from (1).
Indeed it is easily checked with a Mayer--Vietoris argument
that $w^+\in \calI(W^+)$
(by distinguishing the case $b=1$ from the case $c=1$).
Hence we assume in the sequel that $k\geq 3$.

It easily follows from \eqref{eq:Push} that
the $\Push$ map is $\calM(W)$-equivariant in the following sense:
for any $x\in \pi_1\big(UT(\clo{W}),\overrightarrow{\ast}\big)$
and $f\in \calM(W)$, we have
$\Push\big(\overrightarrow{f}_*(x)\big)
= f \Push(x) f^{-1},$
where $\overrightarrow{f}:UT(\clo{W}) \to UT(\clo{W})$ 
is the fiber-bundle map 
induced by the diffeomorphism $\clo{f}:\clo{W} \to \clo{W}$.
Hence we have
\begin{equation} \label{eq:comm}
\big[w^+,\Push^+(u)\big] =
\big(w\, \Push(u)\, w^{-1} \big)^+\, \big(\Push^+(u)\big)^{-1}
= \Push^+\big(\overrightarrow{w}_*(u)\big) \, \big(\Push^+(u)\big)^{-1}.    
\end{equation}
Next, we claim that $U:=u^{-1}\, \overrightarrow{w}_*(u)$
belongs to $\Gamma_k \pi_1\big(UT(\clo{W}),\overrightarrow{\ast}\big)$.
Then, we conclude from~\eqref{eq:comm}  and statement (1) that
$\big[w^+,\Push^+(u)\big] =  \Push^+\big(U)$
belongs to $\Gamma_{k} \calI(W^+)$.

It now remains to prove the above claim.
By the assumption on $w$, we know that
$u^{-1}\, \overrightarrow{w}_*(u) = \tilde{U} \cdot \hbox{fiber}^i$
for some $\tilde U \in \Gamma_k \pi_1\big(UT(\clo{W}),\overrightarrow{\ast}\big)$ and $i\in \Z$ 
(both depending on $u$).
Projecting this identity onto $H_1\big(UT(\clo{W});\Z\big)$,
we obtain that 
$\overrightarrow{w}_*([u]) = [u] + i\cdot  [\hbox{fiber}]$.
But, $w$  acts trivially on $H_1\big(UT(\clo{W});\Z\big)$,
which implies that $i=0$ and proves the claim.

The triviality of the action of $\clo{w}$  
on $H_1\big(UT(\clo{W});\Z\big)$
can be justified as follows,
by distinguishing the case $b=2$ from $b=1$.
Note that, in both cases, we have $\clo{w} \in \calK(\clo{W})$ by our assumption on $k\geq 3$.
If $b=2$, then $\clo{W}$ has one boundary component,
so that we have a short exact sequence
$$
\xymatrix{
0 \ar[r] & H_1(\hbox{fiber}; \Z) \ar[r] &
H_1\big(UT(\clo{W});\Z\big) \ar[r] & H_1(\clo{W};\Z) \ar[r] & 0
}
$$
which splits by choosing a non-singular vector field on $\clo{W}$;
according to \cite{Jo80}, the Johnson subgroup $\calK(\clo{W})$
acts trivially on homotopy classes of non-singular vector fields;
hence, $\clo{w}$ preserves the above splitting
so that the triviality of its  action on $ H_1(\clo{W};\Z)$
implies the same on $H_1\big(UT(\clo{W});\Z\big)$.
If~{$b=1$}, then $\clo{W}$ is closed and
we can find $w'\in \calK(W)$ such that $\clo{w'}=\clo{w}$;
by the same argument as in the previous case applied now to $W$, 
$w'$ acts trivially on $H_1\big(UT({W});\Z\big)$
and, since $H_1\big(UT(\clo{W});\Z\big)$ 
is a quotient of $H_1\big(UT({W});\Z\big)$,
we deduce that $\clo{w}$ acts trivially on $H_1\big(UT(\clo{W});\Z\big)$.
\end{proof}

We now apply Proposition \ref{prop:push-push} in the following two examples, 
where $W^+$ is our reference surface~$\Sigma$ and we take $k:=3$.

\begin{example} \label{ex:c1}
For the decomposition $W \cup V$ of $\Sigma$ shown below,  take $u:=\overrightarrow{\gamma}$
and $w:=T_d$. Then, we have $\Push(u)=T_{\gamma_+} T_{\gamma_-}^{-1}$
and we deduce that 
$\left[T_d,T_{\gamma_+} T_{\gamma_-}^{-1}\right]\in \Gamma_3 \calI(\Sigma)$.\\
$$
\includegraphics[scale=0.55]{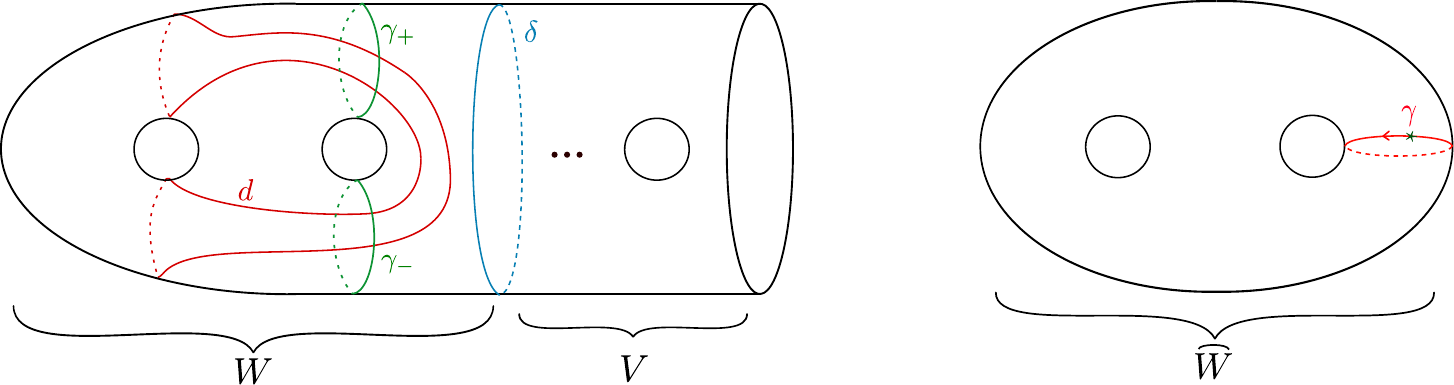} 
$$

\up
\hfill $\blacksquare$
\end{example}

\begin{example}  \label{ex:c2}
For the decomposition $V \cup W$ of $\Sigma$ shown below,  take $u:=\overrightarrow{\gamma}$
and $w:=T_d$. Then, we have $\Push(u)=T_{\gamma_+} T_{\gamma_-}^{-1}$
and we deduce that 
$\left[T_d,T_{\gamma_+} T_{\gamma_-}^{-1}\right]\in \Gamma_3 \calI(\Sigma)$.\\
$$
\includegraphics[scale=0.55]{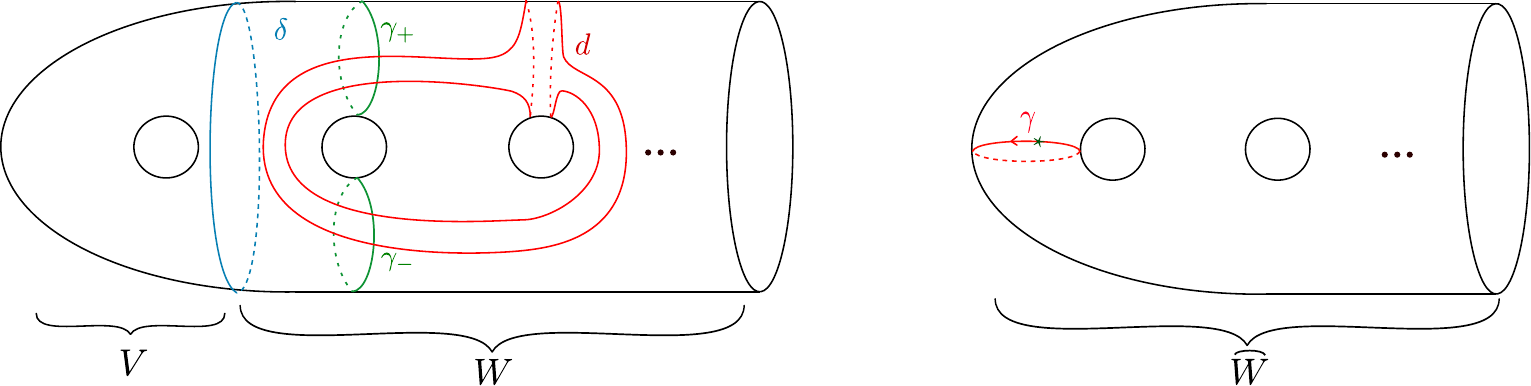} 
$$

\up
\hfill $\blacksquare$
\end{example}

\section{Proof of Theorem~B}

\label{sec:proof_Th_B}

In this section, we define the map $B$ and we prove Theorem~B. 

\subsection{The map $B$}

We embed $\Lambda^4 H^\Q$ into the space $S^2(\Lambda^2 H^\Q)$, 
by sending any $4$-vector $a\wedge b \wedge c \wedge d$
to $(a\wedge b) \cdot (c \wedge d) + (a\wedge c) \cdot (d \wedge b) +
(a\wedge d) \cdot (b \wedge c)$. 
(In the diagrammatic setting that is mentioned below
in Remark \ref{rem:brackets},
this embedding corresponds to the ``IHX'' relations.)
Then we consider the $\Q$-linear map
\begin{equation} \label{eq:B_map}
B=\big(B^{(0)},B^{(2)}\big): \Lambda^2(\Lambda^3 H^\Q) \longrightarrow 
\frac{S^2(\Lambda^2 H^\Q)}{\Lambda^4 H^\Q} \oplus \Q    
\end{equation}
defined by 
\begin{equation} \label{eq:B0}
B^{(0)}\big((x_1\wedge x_2 \wedge x_3)
\wedge (y_1\wedge y_2 \wedge y_3) \big)
= \sum_{i,j\in \Z_3}
\omega(x_i,y_j) \, (x_{i+1} \wedge x_{i+2})
\cdot  (y_{j+1} \wedge y_{j+2}),
\end{equation}
and
\begin{equation} \label{eq:B2}
B^{(2)}\big((x_1\wedge x_2 \wedge x_3)
\wedge (y_1\wedge y_2 \wedge y_3) \big)
= -\frac{1}{4} \begin{vmatrix}
\omega(x_1,y_1) &   \omega(x_1,y_2) & \omega(x_1,y_3)\\
\omega(x_2,y_1) &   \omega(x_2,y_2) & \omega(x_2,y_3)\\
\omega(x_3,y_1) &   \omega(x_3,y_2) & \omega(x_3,y_3)\\
\end{vmatrix}
\end{equation}
We refer to the map $B$ as the \emph{diagrammatic bracket}
since, as specified in the next remark, 
it is also given by  restricting the Lie bracket
of the space of ``symplectic Jacobi diagrams''.

\begin{remark} \label{rem:brackets}
There is a  $\Q$-algebra of ``symplectic Jacobi diagrams'',
denoted by $\calA(H^\Q)$ in \cite{HM09},
whose multiplication is a diagrammatic analogue 
of the Moyal--Weyl product. 
It is a graded Hopf $\Q$-algebra,
whose primitive part $\calA^c(H^\Q)$
consists of \emph{connected} Jacobi diagrams.
It turns out that, in degree $1+1$, its Lie bracket
$$
[-,-]: \Lambda^2 \calA^c_1(H^\Q)  \longrightarrow 
\calA^c_2(H^\Q)
$$
is essentially the above map $B$: 
indeed, according to \cite[Lemma 5.4]{HM09},
the $B^{(0)}$ component corresponds to $0$-looped diagrams (i.e$.$ trees) 
in the target, while the $B^{(2)}$ component corresponds to $2$-looped diagrams 
(i.e$.$ multiples of the theta graph). \hfill $\blacksquare$
\end{remark}

It is proved in \cite[Lemma 5.6 \& (5.5)]{HM09}
that $\ker(B)=R_2^\Q$, i.e$.$ the diagrammatic bracket $B$
and the ``topological'' bracket
$J_2^\Q: \Lie_2(\Lambda^3 H^\Q) 
\to (\Gamma_2\calI/ \Gamma_3 \calI)\otimes \Q$
share the same kernel.
In the sequel, we will just need the inclusion
$R_2^\Q \subset \ker(B)$ which follows also from Proposition \ref{prop:invariants} below.

To state this proposition, we need two ingredients.
On the one hand, we  consider the \emph{second Johnson homomorphism}
$\tau_2:\calK \to {S^2(\Lambda^2 H^\Q)}/{\Lambda^4 H^\Q}$.
According to \cite{Mor89,Mor91}, it is determined by the fact that 
\begin{equation} \label{eq:tau_2}
\tau_2(T_\gamma) 
= \frac{1}{2} \sum_{i,j} (u_i \wedge v_i) \cdot (u_j \wedge v_j) ,
\end{equation}
for any bounding simple closed curve $\gamma$, 
where $(u_1,\dots,u_h,v_1,\dots,v_h)$ is a symplectic basis of
the subsurface of $\Sigma$ bounded  by $\gamma$.

On the other hand, 
we consider the \emph{second core of the Casson invariant} $d'': \calK \to \Q$,
which appeared in Auclair's thesis \cite[Theorem~4.4.6]{Au06}.
This is an alternative to Morita's \emph{core of the Casson invariant} 
$d:\calK \to \Z$ \cite{Mor89,Mor91}. With the knowledge of $\tau_2$,
these two homomorphisms are equivalent one to the other: 
specifically, we have
\begin{equation} \label{eq:d'_d''}
    d=2d'+48d''
\end{equation}
where $d':=-\overline{d'}\circ \tau_2$ and 
$\overline{d'}: S^2(\Lambda^2 H^\Q)/\Lambda^4 H^\Q \to \Q$
is defined by 
$$
\overline{d'}\big((a\wedge b) \cdot (c \wedge d)\big)
= -4 \omega(a,b)\, \omega(c,d) 
-2 \omega(a,c)\, \omega(b,d)  + 2 \omega(a,d)\, \omega(b,c).
$$
Originally,
both homomorphisms are defined from the Casson invariant
and, a posteriori, 
they are determined by the formulas
$$
d(T_\gamma)=4h(h-1) \quad \hbox{and} \quad d''(T_\gamma)=-h/8,
$$
for any bounding simple closed curve $\gamma$,
where $h$ denotes the genus of the subsurface of  $\Sigma$ bounded by $\gamma$.
From our perspectives (see Remark \ref{rem:LMO} below),
we find more natural to work here with $d''$ rather than $d$.

Note that both $\tau_2$ and $d''$ can be restricted to $\Gamma_2 \calI=[\calI,\calI]$,
and they both vanish on $\Gamma_3 \calI$.

\begin{proposition} \label{prop:invariants}
We have $(\tau_2,d'') \circ J_2^\Q = B$.
\end{proposition}

\begin{proof}
That $\tau_2 \circ J_2^\Q=B^{(0)}$ is well-known, and follows from the fact 
that the sequence of  all Johnson  homomorphisms~$\tau_k$ ($k\geq 1$)
defines a Lie algebra map on the associated graded of the ``Johnson filtration'' \cite{Mor93}.

That   $d'' \circ J_2^\Q=B^{(2)}$ is proved as follows.
Let $u,v\in \calI$ and assume, without loss of generality, 
that we have
$$
\tau_1(u) = x_1 \wedge x_2 \wedge x_3
\quad \hbox{and} \quad 
\tau_1(v) = y_1 \wedge y_2 \wedge y_3
$$
for some $x_1,x_2,x_3,y_1,y_2,y_3 \in H$.
We wish to prove that
$$
d''([u,v])= -\frac{1}{4} \begin{vmatrix}
\omega(x_1,y_1) &   \omega(x_1,y_2) & \omega(x_1,y_3)\\
\omega(x_2,y_1) &   \omega(x_2,y_2) & \omega(x_2,y_3)\\
\omega(x_3,y_1) &   \omega(x_3,y_2) & \omega(x_3,y_3)\\
\end{vmatrix}.
$$
But this follows easily from \eqref{eq:d'_d''}, the above definition of $d'$
and Morita's formula \cite[Prop. 5.1]{Mor89} which, in our notations,  
states that
$$
d([u,v]) = 8 \sum_{i,j\in \Z_3}
\omega(x_i,x_{i+1})\, \omega(y_j,y_{j+1})\,  \omega(x_{i+2},y_{j+2}).
$$

\up
\end{proof}

\begin{remark} \label{rem:LMO}
Alternatively, one can obtain Proposition \ref{prop:invariants} 
using  the LMO homomorphism 
$Z: \calI \to \calA(H^\Q)$ that has been introduced in \cite{HM09}.
Indeed, $\tau_2$ coincides with the tree part of the degree 2 part $Z_2$ of $Z$ \cite[Lemma~4.3]{MM13},
while $d''$ is exactly 
the coefficient of the theta graph in $Z_2$ \cite[Lemma~7.3]{MM13}.
Then, the proposition follows directly from \cite[(4.5)]{MM13} 
and Remark \ref{rem:brackets}. \hfill $\blacksquare$
\end{remark}

\subsection{Strategy for the proof of Theorem~B}
\label{subsec:strategy}

We are interested in the submodule
$$
K:= \ker(B) \cap \Lambda^2(\Lambda^3 H) 
$$
of $\Lambda^2(\Lambda^3 H)$.
In the previous subsection, we have justified that $K$ contains 
the kernel of $J:~\Lambda^2(\Lambda^3 H) \to \Gamma_2 \calI/\Gamma_3 \calI$. 
To prove Theorem~B, we now need to show that $J(K)=0$.

In the sequel, we use the following notation and  terminology.
We fix a symplectic basis  $S=~\{a_1,\dots,a_g,b_1,\dots,b_g\}$ 
of $H$ as in \eqref{eq:symplectic_basis}.
An \emph{elementary trivector} is an element 
$s_1 \wedge s_2 \wedge s_3 \in \Lambda^3 H$ with $s_1,s_2,s_3 \in S$.
The wedge of two elementary trivectors is denoted by
$$
\Com{s_1}{s_2}{s_3}{s'_1}{s'_2}{s'_3} :=
(s_1 \wedge s_2 \wedge s_3) \wedge (s'_1 \wedge s'_2 \wedge s'_3) 
\ \in \Lambda^2(\Lambda^3 H).
$$

\begin{remark} \label{rem:choice}
A basis of the free module $\Lambda^2(\Lambda^3 H)$
can be  extracted from the list
$$
\Com{s_1}{s_2}{s_3}{s'_1}{s'_2}{s'_3},
\quad \hbox{with }{s_1,s_2,s_3,s'_1,s'_2,s'_3\in S}
$$
\emph{by deleting zeroes 
and fixing signs among repetitions},
which means the following: if the element
$\Com{s_1}{s_2}{s_3}{s'_1}{s'_2}{s'_3}$ is zero,
then  it  is deleted from the list; otherwise, 
among all the repetitions of 
$\pm \Com{s_1}{s_2}{s_3}{s'_1}{s'_2}{s'_3}$ in the above list,
we choose either  $\Com{s_1}{s_2}{s_3}{s'_1}{s'_2}{s'_3}$
or  $-\Com{s_1}{s_2}{s_3}{s'_1}{s'_2}{s'_3}$. 
Thus, such a basis of $\Lambda^2(\Lambda^3 H)$
is  indexed by unordered pairs $\{I,J\}$
of $3$-element subsets $I,J$ of $S$ such that $I\neq J$
(the corresponding element of the basis
being equal to 
$\pm \big({i_1}\wedge {i_2} \wedge {i_3}\big)
\wedge \big( {j_1}\wedge {j_2} \wedge {j_3}\big)$
if $I,J$ write $\{i_1,i_2,i_3\}$, $\{j_1,j_2,j_3\}$ respectively),
and any two choices of such bases differ by some sign changes.
\hfill $\blacksquare$
\end{remark}

We denote by $\ov{\,\cdot\,}$ the involution of $S$
that is defined by $\ov{a_i}=b_i$ and $\ov{b_i}=a_i$.
A \emph{self-contraction} of an elementary trivector
$s_1 \wedge s_2 \wedge s_3$
is a pair $\{s_i,s_j\}$ such that $s_i= \ov{s_j}$. 
A \emph{mixed contraction} between 
two elementary trivectors $s_1 \wedge s_2 \wedge s_3$ and 
$s'_1 \wedge s'_2 \wedge s'_3$ is a pair $\{s_i,s'_j\}$ such that $s'_j = \ov{s_i}$.
There is the decomposition
$$
\Lambda^2(\Lambda^3 H) = 
V_0 \oplus V_1 \oplus V_2 \oplus V_3
$$
where $V_m$ denotes the submodule generated by the elements
$\Com{s_1}{s_2}{s_3}{s'_1}{s'_2}{s'_3}$
with $m$  mixed contractions between the elementary trivectors
$s_1 \wedge s_2 \wedge s_3$ and $s'_1 \wedge s'_2 \wedge s'_3$. 
Furthermore, $V_1$ decomposes itself as
$$
V_1 = V_{1,0} \oplus V_{1,1} \oplus V_{1,2}
$$
where $V_{1,n}$ is the submodule generated by the elements 
$\Com{s_1}{s_2}{s_3}{s'_1}{s'_2}{s'_3}$
with a single mixed contraction between 
$s_1 \wedge s_2 \wedge s_3$ and $s'_1 \wedge s'_2 \wedge s'_3$,
and a total of $n$ self-contractions
which are disjoint from the mixed contraction. 
Therefore, we have
\begin{equation} \label{eq:U_decomposition}
\Lambda^2(\Lambda^3 H) = 
U_0 \oplus U_1 \oplus U_2 \oplus U_3    
\end{equation}
where $U_0:=V_0$, $U_1:=V_{1,0}$, $U_2:= V_{1,1} \oplus V_2$ 
and $U_3:= V_{1,2} \oplus V_3$.
We can extract 
bases of  $V_m$ and $V_{1,n}$
from the above generating sets
by deleting zeroes and fixing signs among repetitions,
as explained in Remark \ref{rem:choice}.

\begin{lemma} \label{lem:K}
We have 
$
K = U_0 \oplus (K \cap U_1) \oplus 
(K \cap U_2) \oplus (K \cap U_3).
$
\end{lemma}

\begin{proof}
We start by observing that
$$
\frac{S^2(\Lambda^2 H^\Q)}{\Lambda^4 H^\Q}  = W_0 \oplus W_1 \oplus W_{2} 
$$
where $W_r$ denotes the subspace generated by elements of the form
$(s_1 \wedge s_2)\cdot (s_3 \wedge s_4)$
showing exactly $r$  pairs  $\{s_i,s_j\}$
of the form $\{s,\ov{s}\}$ with $s\in S$. If $r=2$, we require that those two pairs are disjoint.

Note that, trivially, $U_0 \subset K$. Let $k\in K$. By \eqref{eq:U_decomposition}, 
there exist $k_0\in U_0,\dots,k_3\in U_3$
such that $k=k_0+\cdots +k_3$. 
Since $B^{(0)}(U_i)$ is contained in $W_{i-1}$ for $i\geq 1$, 
and since
$$
0= B^{(0)}(k)  = B^{(0)}(k_1)+ B^{(0)}(k_2)+ B^{(0)}(k_3), 
$$
each of $B^{(0)}(k_1), B^{(0)}(k_2)$ and $B^{(0)}(k_3)$ is trivial.
Since $B^{(2)}(k_1)=B^{(2)}(k_2)=0$, it follows that $k_i \in K \cap U_i$
for $i=1,2$. Finally, since 
$
0= B^{(2)}(k) =  B^{(2)}(k_3),
$
we have $k_3 \in K \cap U_3$ as well.
\end{proof}

In the sequel, we shall also consider the subgroup
\begin{equation} \label{eq:G}
G:=G(S) \quad 
\big( \simeq \,  \Z_4^g \rtimes   \mathfrak{S}_g\big)    
\end{equation}
of $\Sp(H)$ that is generated by the transformations
$E_i$ and $F_{ij}$, where
\begin{itemize}
\item $E_i$ maps $(a_i,b_i)$ to $(-b_i,a_i)$ 
and fixes all other elements of $S$,
\item $F_{ij}$ exchanges $(a_i,b_i)$ with $(a_j,b_j)$
and fixes all other elements of $S$.
\end{itemize}
Observe that the $G$-action 
preserves the decomposition  \eqref{eq:U_decomposition}
of  $\Lambda^2(\Lambda^3 H)$, 
as well as the resulting decomposition of $K$ in Lemma \ref{lem:K}.
Then, the proof of Theorem~$B$ splits into two parts:
\begin{itemize}
\item in \S \ref{subsec:generating}, we identify  a $G$-generating system  of $K$ 
by finding a $G$-generating system of $K\cap U_i$ 
for each $i\in \{0,1,2,3\}$;
\item in \S \ref{subsec:vanishing}, we show that $J$ vanishes on this $G$-generating system of $K$.
\end{itemize}

\subsection{A finite generating system of $K$} 
\label{subsec:generating}

Set $\ve(s):=\omega(s,\ov{s})$ for all $s\in S$.
It is straightforward to check that
the following elements of 
$\Lambda^2(\Lambda^3 H)$ 
belong to $K=\ker(B) \cap \Lambda^2(\Lambda^3 H) $:
\begin{eqnarray*}
\Diff{x}{y}{c_1}{c_2}{x'}{y'} &:=&
\ve(c_1)\, \Com{x}{y}{c_1}{\ov{c_1}}{x'}{y'}
-\ve(c_2)\, \Com{x}{y}{c_2}{\ov{c_2}}{x'}{y'}\\
&& \hbox{with
$x,y,c_1,c_2,\ov{x'}, \ov{y'}$ pairwise different;}\\
\Square{x}{y}{p}{q}{r}{s} &:=&
\ve(r) \ve(s)\, \Com{x}{p}{q}{y}{\ov{p}}{\ov{q}}
-  \ve(p) \ve(r)\, \Com{x}{q}{s}{y}{\ov{q}}{\ov{s}} \\[-0.2cm]
&& + \ve(p) \ve(q)\, \Com{x}{r}{s}{y}{\ov{r}}{\ov{s}}
- \ve(q)\ve(s)\, \Com{x}{p}{r}{y}{\ov{p}}{\ov{r}} \\
&& \hbox{with 
$p,q,r,s,x,\ov{y}$ pairwise different};\\
\Triple{x}{p}{q}{r}& := &  
\ve(r)\, \Com{x}{p}{q}{x}{\ov{p}}{\ov{q}} - \ve(q)\, \Com{x}{p}{r}{x}{\ov{p}}{\ov{r}} 
+ \ve(p)\, \Com{x}{q}{\ov{r}}{x}{\ov{q}}{{r}}  \\
& &\hbox{with 
$p,\ov{p},q,\ov{q}, r,\ov{r},x$ pairwise different};\\
\IHXone{s_1}{s_2}{s_3}{s_4}{c}&:=&
\Com{s_1}{s_2}{c}{\ov{c}}{s_3}{s_4} 
+ \Com{s_1}{s_3}{c}{\ov{c}}{s_4}{s_2}
+ \Com{s_1}{s_4}{c}{\ov{c}}{s_2}{s_3}\\
&& \hbox{with 
$c,s_2,\ov{s_2},s_3,\ov{s_3},s_4,\ov{s_4}$ pairwise different}\\
&&\hbox{and $s_1 \not\in \{c,s_2,\ov{s_2},s_3,\ov{s_3},s_4,\ov{s_4}\}$;}\\
\IHXtwo{x}{y}{p}{c} & := &  \ve(p)\, \Com{x}{p}{\ov{p}}{y}{p}{\ov{p}} 
- \ve(c)\, \Com{x}{y}{c}{p}{\ov{p}}{\ov{c}} \\
&& \hbox{with 
$x,\ov{x},y,\ov{y},p,\ov{p}$ pairwise different 
and $c\not \in \{x,y,p,\ov{p}\} $;}\\
\IHXthree{p}{q}{r} & := &  
\ve(r) \, \Com{p}{\ov{p}}{q}{\ov{q}}{r}{\ov{r}} 
-\ve(q) \, \Com{r}{\ov{r}}{p}{\ov{p}}{r}{\ov{r}} \\
&& - \ve(r) \, \Com{q}{\ov{q}}{p}{\ov{p}}{r}{\ov{r}} 
+ \ve(p) \, \Com{r}{\ov{r}}{q}{\ov{q}}{r}{\ov{r}} \\
&& \hbox{with 
$p,\ov{p}, q,\ov{q}, r,\ov{r}$ pairwise different};\\
\IHXthreep{p}{q}{r}{s}&:=& 
\ve(q)\,\Com{p}{r}{s}{\ov{p}}{\ov{r}}{\ov{s}}
- \ve(p)\,\Com{q}{r}{s}{\ov{q}}{\ov{r}}{\ov{s}} 
-\ve(s)\,\Com{\ov{r}}{\ov{p}}{q}{{r}}{{p}}{\ov{q}}\\
&& - \ve(r)\,\Com{s}{\ov{p}}{q}{\ov{s}}{p}{\ov{q}} 
-\ve(r)\, \Com{s}{\ov{s}}{q}{\ov{q}}{p}{\ov{p}}
+\ve(s)\, \Com{q}{\ov{q}}{p}{\ov{p}}{r}{\ov{r}}
\\
&&\hbox{with 
$p,\ov{p}, q,\ov{q}, r,\ov{r},s,\ov{s}$ pairwise different}.
\end{eqnarray*}

\begin{theorem} \label{th:K}
As a $G$-submodule of $\Lambda^2(\Lambda^3 H)$,
$K$ is generated by the following 26 elements:
\begin{itemize}
\item[$(\calR_0)$] 
$\Com{a_1}{a_2}{a_3}{a_4}{a_5}{a_6}$,
$\Com{a_1}{b_1}{a_2}{a_3}{a_4}{a_5}$,
$\Com{a_1}{b_1}{a_2}{a_3}{b_3}{a_4}$,\\
$\Com{a_1}{a_2}{a_3}{a_3}{a_4}{a_5}$,
$\Com{a_1}{b_1}{a_2}{a_2}{a_3}{a_4}$,
$\Com{a_1}{b_1}{a_2}{a_2}{a_3}{b_3}$,\\
$\Com{a_1}{a_2}{a_3}{a_2}{a_3}{a_4}$;
\item[$(\calR_1)$]
$\Diff{a_{1}}{a_{2}}{a_5}{a_{3}}{a_{3}}{a_{4}}$,
$\Diff{a_{1}}{a_{2}}{a_{3}}{a_{4}}{a_{3}}{a_{4}}$,
 $\Diff{a_{1}}{a_{2}}{a_{3}}{b_{1}}{a_{3}}{a_{4}}$,\\
 $\Diff{a_{3}}{a_{1}}{a_4}{a_{2}}{a_{2}}{a_{3}}$,
 $\Diff{a_{3}}{a_{1}}{b_{1}}{a_{2}}{a_{2}}{a_{3}}$, 
 $\Diff{a_1}{a_2}{a_4}{a_3}{a_2}{a_{1}}$,\\
 $\IHXone{a_{1}}{a_{2}}{a_{3}}{a_{4}}{b_{1}}$;
\item[$(\calR_2)$] 
$\Square{a_1}{a_2}{a_4}{a_5}{a_3}{b_3}$,
$\Square{a_1}{a_2}{a_4}{b_4}{a_3}{b_3}$,
$\Square{a_1}{a_2}{a_2}{b_1}{a_3}{b_3}$,\\
$\Square{a_1}{a_2}{a_2}{a_4}{a_3}{b_3}$,
$\Square{a_1}{a_2}{b_1}{a_4}{a_3}{b_3}$,
$\Triple{a_1}{a_2}{a_3}{a_4}$,\\
$\IHXtwo{a_1}{a_2}{a_3}{a_4}$, $\IHXtwo{a_1}{a_2}{a_3}{b_1}$;
\item[$(\calR_3)$] 
$\Diff{a_1}{b_1}{a_2}{b_2}{a_3}{b_3}$,
$\Diff{a_1}{b_1}{a_2}{a_3}{a_4}{b_4}$,
 $\IHXthree{a_1}{a_2}{a_3}$,
  $\IHXthreep{a_1}{a_2}{a_3}{a_4}$.
\end{itemize}
\end{theorem}

\noindent
In the above statement, 
it is understood that we only consider
the elements making sense 
in the given genus:
for instance, the element $\Com{a_1}{a_2}{a_3}{a_4}{a_5}{a_6}$
in ($\calR_0)$ has to be removed
in the cases  $g<6$.
(Recall that our permanent assumption is that $g\geq 3$.)

The proof of Theorem~\ref{th:K}, which is given below,
consists in proving that the family ($\calR_i)$ 
is a $G$-generating system of $K\cap U_i$  
for each case $i=0$, $i=1$, $i=2$ and $i=3$.
Actually, as we shall see, 
the proofs are rather different in nature 
from one case to another.

\subsubsection{A $G$-generating system  of $K \cap U_0$.}

By its very definition, the module
$K\cap U_0=U_0=V_0$ is generated by wedge of elementary trivectors
without mixed contractions, i.e$.$ the elements
\begin{equation} \label{eq:no_contraction}
\Com{s_1}{s_2}{s_3}{s'_1}{s'_2}{s'_3} \quad
\hbox{with  $\{\ov{s_1},\ov{s_2},\ov{s_3}\}
\cap \{s'_1,s'_2,s'_3\} =\varnothing$.}
\end{equation}

Let $N:S \to \{1,\dots,g\}$ be the map defined by $N(a_i)=N(b_i)=i$.
We consider an  element  
$\Lambda:= \Com{s_1}{s_2}{s_3}{s'_1}{s'_2}{s'_3}$
of the form \eqref{eq:no_contraction},  and let 
$$
n:= \sharp\ N(\{s_1,s_2,s_3\})\cap N(\{s_1',s_2',s'_3\}).
$$
If $n=0$, then $\pm \Lambda$ is in the $G$-orbit of 
$\Com{a_1}{a_2}{a_3}{a_4}{a_5}{a_6}$,
or $\Com{a_1}{b_1}{a_2}{a_3}{a_4}{a_5}$,
or $\Com{a_1}{b_1}{a_2}{a_3}{b_3}{a_4}$.
If $n=1$, then  
$\pm \Lambda$ is in the $G$-orbit of 
$\Com{a_1}{a_2}{a_3}{a_3}{a_4}{a_5}$
or $\Com{a_1}{b_1}{a_2}{a_2}{a_3}{a_4}$,
or $\Com{a_1}{b_1}{a_2}{a_2}{a_3}{b_3}$.
If $n=2$, then  $\pm \Lambda$ is in the $G$-orbit of 
$\Com{a_1}{a_2}{a_3}{a_2}{a_3}{a_4}$.
Thus, $(\calR_0)$
is a $G$-generating system of $K \cap U_0$.

\subsubsection{A $G$-generating system  of $K \cap U_1$.}

As in the proof of Lemma \ref{lem:K}, 
we consider here the subspace 
$W_0$ of $S^2(\Lambda^2 H^\Q)/\Lambda^4 H^\Q$
generated by the elements $(s_1\wedge s_2)\cdot (s_3\wedge s_4)$
showing no contraction. This decomposes as a direct sum
$$
W_0 = \bigoplus_{\{s_1,s_2,s_3,s_4\}} W_{s_1,s_2,s_3,s_4}
$$
indexed by the unordered quadruplets 
$\{s_1,s_2,s_3,s_4\}\in S^4/\mathfrak{S}_4$ 
showing no contraction and at most 2 repetitions, where 
$W_{s_1,s_2,s_3,s_4}$ is the subspace
generated by $(s_{\sigma(1)}\wedge s_{\sigma(2)})
\cdot (s_{\sigma(3)}\wedge s_{\sigma(4)})$ for 
$\sigma \in \mathfrak{S}_4$.
Similarly, the submodule $U_1=V_{1,0}$ of $\Lambda^2(\Lambda^3 H)$
decomposes as a direct sum 
$$
U_1 = \bigoplus_{\{s_1,s_2,s_3,s_4\}} U_{s_1,s_2,s_3,s_4}
$$
indexed by the same quadruplets,
where we denote by $U_{s_1,s_2,s_3,s_4}$ the submodule generated
by $\Com{s_{\sigma(1)}}{s_{\sigma(2)}}{a_i}
{b_i}{s_{\sigma(3)}}{s_{\sigma(4)}}$ for $i\in\{1,\dots,g\}$
and $\sigma \in \mathfrak{S}_4$.
The map $B^{(0)}: U_1 \to W_0$ preserves the above decompositions,
so that we have 
$$
K\cap U_1 
= \bigoplus_{\{s_1,s_2,s_3,s_4\}} K \cap U_{s_1,s_2,s_3,s_4}.
$$
Moreover, the map $B^{(0)}$ is $G$-equivariant:
so,  a generating system of the $G$-module $K\cap U_1$
is obtained by determining a generating system 
of the module $K \cap U_{s_1,s_2,s_3,s_4}$ for a 
representative $\{s_1,s_2,s_3,s_4\}$ in each $G$-orbit. 
But, an unordered quadruplet  $\{s_1,s_2,s_3,s_4\}$ 
with no contraction and at most $2$ repetitions 
is $G$-equivalent to either
$\{a_1,a_2,a_3,a_4\}$, or  $\{a_1,a_2,a_3,a_3\}$, or 
$\{a_1,a_1,a_2,a_2\}$.

The following three lemmas deal with the above three subcases
in this order and, all together, they prove that $(\calR_1)$
is a $G$-generating system of $K \cap U_1$.

\begin{lemma} \label{lem:a_1,a_2,a_3,a_4}
If $g\geq 4$, the module $K\cap U_{a_1,a_2,a_3,a_4}$
is generated by
\begin{itemize}
\item[(i)] $\Diff{a_{\sigma(1)}}{a_{\sigma(2)}}{a_i}{a_{\sigma(3)}}
{a_{\sigma(3)}}{a_{\sigma(4)}}$
with $i\geq 5$, $\sigma \in \mathfrak{S}_4$,
\item[(ii)] $\Diff{a_{\sigma(1)}}{a_{\sigma(2)}}{a_{\sigma(3)}}
{a_{\sigma(4)}}{a_{\sigma(3)}}{a_{\sigma(4)}}$
with $\sigma \in \mathfrak{S}_4$,
\item[(iii)] $\Diff{a_{\sigma(1)}}{a_{\sigma(2)}}{a_{\sigma(3)}}
{b_{\sigma(1)}}{a_{\sigma(3)}}{a_{\sigma(4)}}$
with $\sigma \in \mathfrak{S}_4$,
\item[(iv)] $\IHXone{a_{\sigma(1)}}{a_{\sigma(2)}}{a_{\sigma(3)}}
{a_{\sigma(4)}}{b_{\sigma(1)}}$
with $\sigma \in \mathfrak{S}_4$.
\end{itemize}
\end{lemma}

\begin{lemma} \label{lem:a_1,a_2,a_3,a_3}
The module $K \cap U_{a_1,a_2,a_3,a_3}$ is generated by
\begin{itemize}
\item[(i)] $\Diff{a_{3}}{a_{\sigma(1)}}{a_i}{a_{\sigma(2)}}
{a_{\sigma(2)}}{a_{3}}$
with $i\geq 4$, $\sigma \in \mathfrak{S}_2$,
\item[(ii)] 
$\Diff{a_{3}}{a_{\sigma(1)}}{b_{\sigma(1)}}{a_{\sigma(2)}}
{a_{\sigma(2)}}{a_{3}}$
with  $\sigma \in \mathfrak{S}_2$.
\end{itemize}
\end{lemma}

\begin{lemma} \label{lem:a_1,a_1,a_2,a_2}
The module $K\cap U_{a_1,a_1,a_2,a_2}$
is generated by  
$\Diff{a_1}{a_2}{a_i} {a_3}{a_2}{a_{1}}$ with $i\geq 4$.
\end{lemma}

\begin{proof}[Proof of Lemma \ref{lem:a_1,a_2,a_3,a_4}]
The module $U:=U_{a_1,a_2,a_3,a_4}$ is free, 
and a basis is obtained
from the following list by deleting zeroes and fixing signs
among repetitions, as explained in Remark \ref{rem:choice}:
\begin{equation} \label{eq:basis}
\Com{a_{\sigma(1)}}{a_{\sigma(2)}}{a_i}
{b_i}{a_{\sigma(3)}}{a_{\sigma(4)}}
\quad \hbox{with }
 \sigma \in \mathfrak{S}_4,  i\in \{1,\dots,g\}.    
\end{equation}
We choose such a basis of $U$ 
and we denote it by $\{e_{J,i}\}_{J,i}$,
which is indexed by $2$-element subsets  $J\subset \{1,2,3,4\}$
together with an element $i \in \{1,\dots, g\}\setminus J$;
specifically, we have 
$e_{J,i}= \pm 
\Com{a_{j}}{a_{j'}}{a_i}{b_i}{a_{k}}{a_{k'}}$
writing $J$ 
and $\overline{J}=\{1,2,3,4\}\setminus J$
as $\{j,j'\}$ and $\{k,k'\}$,
respectively.
Let $U'$ be the submodule 
of  $U$ (freely) generated by the previous elements $e_{J,i}$
for $i\leq 4$
(hence $J\subset \{1,2,3,4\}$ is a $2$-element subset
and $i\in \overline{J}$).
When viewed as a relation in $U$, 
any element of type~(i)
identifies a basis element $e_{J,i}$
with $\pm e_{J,k}$
for  $i\geq 5$, $J\subset \{1,2,3,4\}$ of cardinality $2$
and some $k\in \overline{J}$.
Similarly, 
for any  $J\subset \{1,2,3,4\}$ of cardinality $2$,
writing  $\overline{J}=\{i,k\}$,
the basis element $e_{J,i}$
is identified with $\pm e_{J,k}$
if an appropriate element of type (ii) is viewed as a relation in $U$.
Therefore, the inclusion of $U'$  in $U$ induces an isomorphism 
$$
\frac{U'}{\langle \,\hbox{(ii)}, \hbox{(iii)},
\hbox{(iv)} \, \rangle} \stackrel{\simeq}{\longrightarrow}
\frac{U}
{\langle \, \hbox{(i)},\hbox{(ii)}, \hbox{(iii)},
\hbox{(iv)} \, \rangle}.
$$
 
 So, we can assume for the rest of the proof that $g=4$,
and we  need to show that the map $B^{(0)}$
induces an injection from 
$U/{\langle \, \hbox{(ii)}, \hbox{(iii)}, \hbox{(iv)} \, \rangle}$
to ${S^2(\Lambda^2 H^\Q)}/{\Lambda^4 H^\Q}$.
By the above discussion,
the module
$U/{\langle \, \hbox{(ii)} \, \rangle}$ is free of rank 6,
with basis $\{e_J\}_J$ indexed 
by $2$-element subsets $J$ of $\{1,2,3,4\}$ where
$e_J = \pm
\Com{a_{j}}{a_{j'}}{a_i}{b_i}{a_{i}}{a_{k}}$
writing $J=\{j,j'\}$ and $\overline{J}=\{i,k\}$.
When viewed as a  relation 
in $U/{\langle \, \hbox{(ii)} \, \rangle}$,
any element of type (iii)
identifies a basis element $e_J$ with $\pm e_{\overline{J}}$
for a $2$-element subset $J$ of $\{1,2,3,4\}$.
Consequently, 
the module 
$U/{\langle \, \hbox{(ii)}, \hbox{(iii)}\, \rangle}$
is free of rank $3$, with basis $\{e_P\}_P$
indexed by partitions $P=J \sqcup \overline{J}$ of 
$\{1,2,3,4\}$ into two $2$-element subsets.

On the other hand, the $\Q$-vector space $W:=W_{a_1,a_2,a_3,a_4}$ 
is generated by  
 \begin{equation} \label{eq:family_2}
\big\{(a_1 \wedge a_2) \cdot (a_3 \wedge a_4)\,,\,
(a_1 \wedge a_3) \cdot (a_4 \wedge a_2)\,,\,
(a_1 \wedge a_4) \cdot (a_2 \wedge a_3)\big\}
\end{equation}
with the single relation
\begin{equation} \label{eq:relation}
(a_1 \wedge a_2) \cdot (a_3 \wedge a_4)+ 
(a_1 \wedge a_3) \cdot (a_4 \wedge a_2) +
(a_1 \wedge a_4) \cdot (a_2 \wedge a_3).    
\end{equation}
The map $B^{(0)}:
U/{\langle \, \hbox{(ii)}, \hbox{(iii)} \, \rangle} \to W$
sends the basis element $e_{P}$ to 
$\pm (a_{j} \wedge a_{j'})\cdot (a_{k} \wedge a_{k'})$
writing $P=\{j,j'\}\sqcup \{k,k'\}$. Finally,
when viewed as a relation in 
$U/{\langle \, \hbox{(ii)}, \hbox{(iii)} \, \rangle}$,
any element of type (iv)  goes to \eqref{eq:relation}.
The conclusion follows.
\end{proof}

\begin{proof}[Proof of Lemma \ref{lem:a_1,a_2,a_3,a_3}]
The module $U:=U_{a_1,a_2,a_3,a_3}$ is free, 
and a basis is obtained
from the following list by deleting zeroes and fixing signs
among repetitions, as explained in Remark \ref{rem:choice}:
\begin{eqnarray*} 
\Com{a_{\sigma(1)}}{a_{\sigma(2)}}{a_i}
{b_i}{a_{\sigma(3)}}{a_{\sigma(4)}}
&& \hbox{with }
\{1,2,3,4\} \stackrel{\sigma}{\longrightarrow} \{1,2,3\}
\ \hbox{\small surjective, reaching 3 twice}, \\
&& \hbox{\small and } i\in \{1,\dots,g\}.    
\end{eqnarray*}
We choose such a basis of $U$
and denote it by $\{e_{j,i}\}_{j,i}$,
which is indexed by $j\in \{1,2\}$ and $i\not \in \{j,3\}$;
specifically, we have 
$e_{j,i} = \pm \Com{a_3}{a_j}{a_i}{b_i}{a_{\overline{j}}}{a_3}$
where $\overline{j}$ is the element of $\{1,2\}$ not equal to~$j$.
When viewed as a relation in $U$, 
any element of type (i) identifies a basis element
$e_{j,i}$ with $\pm e_{j,\overline{j}}$
for $i\geq 4$ and $j\in \{1,2\}$,
and any element of type (ii) identifies $e_{1,2}$
with $\pm e_{2,1}$.
It follows that the module
$U/{\langle \, \hbox{(i)}, \hbox{(ii)} \, \rangle}$
is free of rank $1$. 

On the other hand, the $\Q$-vector space $W:=W_{a_1,a_2,a_3,a_3}$
is $1$-dimensional, generated by 
$(a_1\wedge a_3) \cdot (a_2 \wedge a_3)$.
Since $B^{(0)}$ maps $e_{1,2}$ to 
$\pm (a_1\wedge a_3) \cdot (a_2 \wedge a_3)$,
it induces an injection 
$U/{\langle \, \hbox{(i)}, \hbox{(ii)} \, \rangle} \to W$.
\end{proof}

\begin{proof}[Proof of Lemma \ref{lem:a_1,a_1,a_2,a_2}]
The module $U:=U_{a_1,a_1,a_2,a_2}$ is free, 
and a basis is obtained
from the following list by deleting zeroes and fixing signs
among repetitions, as explained in Remark \ref{rem:choice}:
\begin{eqnarray*} 
\Com{a_{\sigma(1)}}{a_{\sigma(2)}}{a_i}
{b_i}{a_{\sigma(3)}}{a_{\sigma(4)}}
&& \hbox{with }
\{1,2,3,4\} \stackrel{\sigma}{\longrightarrow} \{1,2\}
\ \hbox{\small  reaching each of 1 and 2 twice}, \\
&& \hbox{\small and } i\in \{1,\dots,g\}.    
\end{eqnarray*}
We choose such a basis of $U$ and denote it 
by $\{e_i\}_{i\geq 3}$,
having 
$e_i=\pm \Com{a_1}{a_2}{a_i}{b_i}{a_2}{a_1}$.
When viewed as a relation in $U$, any element of type (i)
identifies a basis element $e_i$ with $\pm e_3$ for $i\geq 3$.
It follows that the module
$U/{\langle \, \hbox{(i)} \, \rangle}$
is free of rank $1$. 

Besides, the $\Q$-vector space $W:=W_{a_1,a_1,a_2,a_2}$
is $1$-dimensional, generated by 
$(a_1\wedge a_2) \cdot (a_1 \wedge a_2)$.
Since $B^{(0)}$ maps $e_{3}$ to 
$\pm (a_1\wedge a_2) \cdot (a_1 \wedge a_2)$,
it induces an injection 
$U/{\langle \, \hbox{(i)}  \, \rangle} \to W$.
\end{proof}

\subsubsection{A $G$-generating system of $K \cap U_2$.}

As in the proof of Lemma \ref{lem:K}, 
we consider now the subspace 
$W_1$ of $S^2(\Lambda^2 H^\Q)/\Lambda^4 H^\Q$
generated by the elements $(s_1\wedge s_2)\cdot (s_3\wedge s_4)$
showing a single contraction. This decomposes as a direct sum
$$
W_1 = \bigoplus_{\{s,s'\}} W_{s,s'}
$$
indexed by the unordered pairs $\{s,s'\}\in S^2/\mathfrak{S}_2$
such that $s'\neq \ov{s}$,
where $W_{s,s'}$ denotes the subspace generated
by the elements
$(s \wedge x)\cdot (s'\wedge \ov{x})$
for all $x\in S$.
Similarly, the submodule $U_2=V_{1,1} \oplus V_2$
decomposes as a direct sum
$$
U_2 =  \bigoplus_{\{s,s'\}} U_{s,s'}
$$
indexed by the same pairs, 
where $U_{s,s'}$ denotes the submodule generated by the elements
$\Com{s}{x}{y}{s'}{\ov{x}}{\ov{y}}$ for all $x,y\in S$
and
$\Com{s}{s'}{p}{\ov{p}}{q}{\ov{q}}$
for all $p\in S,q \in S\setminus \{s,s',\ov{s},\ov{s'}\}$.
 The map $B^{(0)}:U_2 \to W_1$
 preserves the above decompositions, so that we have
$$
K \cap U_2 =  \bigoplus_{\{s,s'\}} K \cap U_{s,s'}.
$$
Since $B^{(0)}$ is $G$-equivariant, 
a $G$-generating system of $K\cap U_2$
is obtained by determining a generating system
of $K \cap U_{s,s'}$ for a representative $\{s,s'\}$ 
in each $G$-orbit.
Thus, we are reduced to determine a generating system
of $K\cap U_{a_1,a_2}$ and $K\cap U_{a_1,a_1}$.
This is achieved by the following two lemmas
which prove that $(\calR_2)$
is a $G$-generating system of $K \cap U_2$.

\begin{lemma} \label{lem:a_1,a_2}
The module $K \cap U_{a_1,a_2}$ is generated by
\begin{itemize}
\item[(i)] 
$\Square{a_{\sigma(1)}}{a_{\sigma(2)}}{p}{q}{r}{\ov{r}}$
with $\sigma\in \mathfrak{S}_2$, 
$r\not\in \{p,\ov{p},q,\ov{q},a_1,a_2,b_1,b_2\}$
and $p,q,a_{\sigma(1)},b_{\sigma(2)}$ pairwise different,
\item[(ii)] 
$\IHXtwo{a_{\sigma(1)}}{a_{\sigma(2)}}{p}{q}$
with $\sigma\in \mathfrak{S}_2$, 
$p\not\in \{a_1,a_2,b_1,b_2\}$
and $q \not\in \{ a_{1}, a_{2}, p, \ov{p}\}$.
\end{itemize}
\end{lemma}

\begin{lemma} \label{lem:a_1,a_1}
The module $K \cap U_{a_1,a_1}$ is generated by
\begin{itemize}
\item[$(*)$] $\Triple{a_1}{p}{q}{r}$
with $p,q,r \not\in \{a_1,b_1\}$ and
$p,\ov{p},q,\ov{q},r,\ov{r}$ pairwise different.
\end{itemize}
\end{lemma}

\begin{proof}[Proof of Lemma \ref{lem:a_1,a_2}]
The module $U:=U_{a_1,a_2}$ is free of basis
 $\{u(x,y)\}_{x,y} \sqcup \{u'(q,p)\}_{q,p}$ where
\begin{equation} \label{eq:basis_u}
u(x,y) := \Com{a_1}{x}{y}{a_2}{\ov{x}}{\ov{y}}
\hbox{}
\end{equation}
is indexed by 2-element subsets $\{x,y\}$ 
of $S\setminus \{a_1,b_2\}$, and 
\begin{equation} \label{eq:basis_up}
u'(q,p) := \Com{a_1}{a_2}{q}{\ov{q}}{p}{\ov{p}}
\end{equation}
is indexed by pairs $(q,p)$ 
with $q\in S\setminus\{a_1,a_2\}$
and $p \in S\setminus\{a_1,a_2,b_1,b_2, q,\ov{q}\}$.
We claim that the quotient module 
$U/{\langle \, \hbox{(i), (ii)}  \, \rangle}$
is generated by 
\begin{equation} \label{eq:gen_U}
 u(a_3,b_3) \quad \hbox{and} \quad
u(a_2,p) \hbox{ for $p\in S \setminus \{a_1,a_2,b_2\}$}.
\end{equation}

We start by observing that any element $u'(q,p)$ 
of type  \eqref{eq:basis_up} is, modulo (ii),
equal to the element $\pm u(p,\ov{p})$ of type \eqref{eq:basis_u}.
Indeed, we have 
$$
\IHXtwo{a_1}{a_2}{p}{q}=
-\ve(p)\, u(p,\ov{p}) - \ve(q)\, u'(q,p).
$$
Thus, to prove the above claim, it is enough to show
 that the elements \eqref{eq:basis_u} 
not among the elements \eqref{eq:gen_U} 
can be written in terms of the latter using the relations (i).
For any $r\in S \setminus\{a_1,a_2,a_3,b_1,b_2,b_3\}$, 
we have the following elements of type (i):
$$
\Square{a_{1}}{a_{2}}{r}{a_2}{a_3}{b_3}
=- u(r,a_2) -\ve(r)\, u(a_2,b_3) 
+ \ve(r)\,  u(a_3,b_3) +  u(r,a_3)
$$
$$
\Square{a_{1}}{a_{2}}{r}{a_2}{b_3}{a_3}
= - u(r,a_2) + \ve(r)\, u(a_2,a_3) 
+ \ve(r)\, u(b_3,a_3) - u(r,b_3)
$$
$$
\Square{a_{1}}{a_{2}}{a_3}{b_3}{r}{\ov{r}}
= - u(a_3,b_3) - \ve(r)\, u(b_3,\ov{r})
- u(r,\ov{r}) - \ve(r)\, u(a_3,r)
$$
We deduce that, modulo (i), $u(r,\ov{r})$ 
is a linear combination of \eqref{eq:gen_U}
for $r\in S \setminus\{a_1,a_2,b_1,b_2\}$.
Next, for any such $r$ and $p\in S \setminus \{a_1,a_2,b_2, r,\ov{r}\} $, considering 
$$
\Square{a_{1}}{a_{2}}{p}{a_2}{r}{\ov{r}}
= - u(p,a_2) - \ve(p)\ve(r)\, u(a_2,\ov{r}) 
+ \ve(p)\, u(r,\ov{r}) +\ve(r)\,  u(p,r)
$$
shows that $u(p,r)$
is also modulo (i) a linear combination of \eqref{eq:gen_U}.
This proves the above claim.

The $\Q$-vector space $W:=W_{a_1,a_2}$ has a basis $\{w(x)\}_x$ defined by
$w(x) := (a_1\wedge x )\cdot (a_2\wedge \ov{x})$ for $x\in S \setminus \{a_1,b_2\}$.
The map  $B^{(0)}: U/{\langle \, \hbox{(i)},\hbox{(ii)}  \, \rangle} \to W$ 
sends $u(x,y)$ to $\ve(x)\, w(y)+ \ve(y)\, w(x)$. 
Hence, the images of the generators \eqref{eq:gen_U} are
$$
-w(a_3) + w(b_3), \quad 
\ve(p)\, w(a_2)+ w(p)  \
\hbox{ for $p\in S \setminus \{a_1,a_2,b_2\}$},
$$
and they are easily checked to be $\Q$-linearly independent.
It follows that the map  
$B^{(0)}: U/{\langle \, \hbox{(i)}, \hbox{(ii)}  \, \rangle} \to W$ 
is injective.
\end{proof}

\begin{proof}[Proof of Lemma \ref{lem:a_1,a_1}]
The module $U:=U_{a_1,a_1}$ is free, and a basis
is obtained from the following list
by deleting zeroes and 
fixing signs among repetitions: 
\begin{equation} \label{eq:basis_u*}
 \Com{a_1}{x}{y}{a_1}{\ov{x}}{\ov{y}}
\quad \hbox{with $x,y \in S\setminus \{a_1,b_1\}$}
\end{equation}
Thus, we get a basis $\{u(x,y)\}_{x,y}$ indexed 
by  unordered pairs $\{x,y\}\subset S\setminus\{a_1,b_1\}$
such that $x\neq \ov{y}$, 
up to the involution 
$\{x,y\} \mapsto \{\ov{x},\ov{y}\}$, 
and where $u(x,y)= \pm \Com{a_1}{x}{y}{a_1}{\ov{x}}{\ov{y}}$.

For any $q,r\in \{a_3,\dots,a_g,b_3,\dots,b_g\}$
with $q\neq \ov{r}$,
we have the following element of type $(*)$:
$$
\Triple{a_1}{a_2}{q}{r} = 
\pm u(a_2,q) \pm u(a_2,r) \pm u(q,\ov{r}).
$$
Hence, $U/\langle \,(*)\, \rangle$ is generated
by the elements $u(x,y)$ such that the pair
$\{x,y\}$ contains $a_2$ or $b_2$.
Furthermore, for any  $r\in \{b_4,\dots,b_g\}$, the elements
$$
\Triple{a_1}{a_2}{b_3}{r} 
= \pm u(a_2,b_3) \pm u(a_2,r) \pm u(b_3,\ov{r})
$$
$$
\Triple{a_1}{\ov{r}}{a_2}{b_3} 
= \pm u(\ov{r},a_2) \pm u(\ov{r},b_3) \pm u(a_2,a_3)
$$
show that, modulo $(*)$, $ u(a_2,r)$ is a linear combination of 
\begin{equation} \label{eq:gen_U*}
u(a_2,q) \quad  \hbox{ for } q\in \{a_3,a_4,\dots,a_g,b_3\}.
\end{equation}
Hence the quotient module 
$U/\langle \,(*)\, \rangle$ 
is generated by \eqref{eq:gen_U*}.

The $\Q$-vector space $W:=W_{a_1,a_1}$ has the basis 
$\{w(x)\}_x$ defined by
$w(x) := (a_1\wedge x )\cdot (a_1\wedge \ov{x})$ 
and indexed by  the elements $x\in S \setminus \{a_1,b_1\}$,
up to the involution $x\mapsto \ov{x}$.
The map  $B^{(0)}: U/{\langle \, (*) \, \rangle} \to W$ 
sends $u(x,y)$ to $\pm \big(\ve(x)\, w(y)+ \ve(y)\, w(x)\big)$. 
Hence, the images of the generators \eqref{eq:gen_U*} are
$$
\pm\big (\ve(q)\, w(a_2)+ w(q) \big)  \
\hbox{ for } q\in \{a_3,a_4, \dots,a_g,b_3\}.
$$
Those vectors being  linearly independent, the map 
$B^{(0)}: U/{\langle \, (*) \, \rangle} \to W$  is injective.
\end{proof}

\subsubsection{A $G$-generating system  of $K \cap U_3$.}

As in the proof of Lemma \ref{lem:K}, 
we finally consider the subspace 
$W_2$ of $S^2(\Lambda^2 H^\Q)/\Lambda^4 H^\Q$
generated by the elements $(s_1\wedge s_2)\cdot (s_3\wedge s_4)$
with $s_i\in S$, showing two (disjoint) contractions. We are interested in the kernel $K \cap U_3$ of the map
$$
B=\big(B^{(0)},B^{(2)}\big):U_3 \longrightarrow W_2 \oplus \Q
$$
defined on the module $U_3 = V_{1,2} \oplus V_3$.
Note that, in contrast with the previous cases, 
the component $B^{(2)}$ of $B$ plays a role here, although minor.

The next lemma proves that 
$(\calR_3)$ is a $G$-generating system of $K \cap U_3$.

\begin{lemma} \label{lem:K,U_3}
The module $K\cap U_3$ is generated by
\begin{itemize}
\item [(i)] $\Diff{x}{\ov{x}}{c}{\ov{c}}{y}{\ov{y}}$
with $x,\ov{x},y,\ov{y},c,\ov{c}$ pairwise different,
\item [(ii)] $\Diff{x}{\ov{x}}{c_1}{c_2}{y}{\ov{y}}$
with $x,\ov{x},y,\ov{y},c_1,\ov{c_1},c_2,\ov{c_2}$ pairwise different,
\item [(iii)] $\IHXthree{p}{q}{r}$
with $p,\ov{p},q,\ov{q},r,\ov{r}$ pairwise different,
\item [(iv)] $\IHXthreep{p}{q}{r}{s}$
with $p,\ov{p},q,\ov{q},r,\ov{r},s,\ov{s}$ pairwise different.
\end{itemize}
\end{lemma}

\begin{proof}
We endow $U:=U_3$ and $W:=W_2$
with the following filtrations.
First, observe that the map  $N:S \to \{1,\dots,g\}$  
defined by $N(a_i)=N(b_i)=i$
induces an increasing filtration 
$$
\varnothing = F_0 S \subset F_1 S
\subset \cdots \subset F_{g-1} S \subset F_g S=S
$$
of the set $S$, where
$F_k S$ consists of the $s\in S$ with $N(s)\leq k$.
Thus, we get increasing filtrations
\begin{eqnarray*}
&& \{0\} = F_0U \subset F_1 U
\subset \cdots \subset F_{g-1} U \subset F_g U=U\\
&& \{0\} = F_0W \subset F_1 W
\subset \cdots \subset F_{g-1} W \subset F_g W=W
\end{eqnarray*}
of the module $U$ and the $\Q$-vector space $W$, respectively.
Specifically, for every $k\in\{1,\dots,g\}$,
the module $F_k U$ is generated by the  elements
\begin{equation} \label{eq:uuuu}
u(p,q,r) := \Com{p}{q}{r}{\ov{p}}{\ov{q}}{\ov{r}}
\quad  \hbox{with $p,q,r\in F_k S$ pairwise different},
\end{equation}
and 
\begin{equation} \label{eq:uuuu'}
u'(r;p,q):= \Com{p}{\ov{p}}{r}{\ov{r}}{q}{\ov{q}}
\quad  \hbox{with $p,q,r\in F_k S$ 
and $p,\ov{p},q,\ov{q},r,\ov{r}$ pairwise different}.
\end{equation}
Similarly, for every $k\in\{1,\dots,g\}$,
the $\Q$-vector space $F_kW$
is generated by the elements
$$
w(p,q) := (p \wedge q) \cdot  (\ov{p} \wedge \ov{q})
\quad  \hbox{with $p,q\in F_k S$ and $p\neq q$}.
$$
For future use, we compute the map $B^{(0)}: U \to W$
in terms of those generators:
\begin{eqnarray}
\label{eq:val_1} B^{(0)}\big({u}(p,q,r)\big)
&=& 
\ve(p)\, w(q,r) + \ve(q)\, w(p,r)+ \ve(r)\, w(p,q),
\\
\label{eq:val_2} 
B^{(0)}\big({u'}(r;p,q) \big)
&=&  \ve(r)\, w(p,q)-  \ve(r)\, w(p,\ov{q}).
\end{eqnarray}

Since we have $B^{(0)}(F_kU)\subset F_k W $,
there is a  map 
$\Gr B^{(0)}: \Gr U \to \Gr W$
on the associated graded,
and we set $L:= \ker \big(\Gr B^{(0)}\big)$.
The filtration on $U$ induces a filtration on $K\cap U$, 
by setting $F_k(K \cap U) :=  K \cap (F_k U)$.
By definition of $K$, the inclusion of $K\cap U$ in $U$
induces an injective map
$$
\kappa: \Gr(K \cap U) \longrightarrow L.
$$
Let $I$ be the submodule of $U$
generated by the elements (i),(ii),(iii),(iv) of the statement,
and endow $I$ with the filtration 
defined by $F_k I := I \cap (F_k U)$.
The inclusion of $I$ in $K \cap U$ induces an injective map
$$
\iota: \Gr I \longrightarrow \Gr(K \cap U).
$$
We claim two things:
\begin{eqnarray}
\label{eq:claim_1}&&
\hbox{the composition  
$\kappa_k \circ \iota_k:\Gr_k I \to L_k$ 
is surjective  for $k\geq 4$;}\\
\label{eq:claim_2}&& \hbox{we have $F_3 I =F_3( K\cap U)$.}
\end{eqnarray}
Claim \eqref{eq:claim_1} will imply  that 
$\iota_k$ is surjective for every $k\geq 4$,
from which we  will deduce that $K\cap U=I + F_3(K\cap U)$;
then, with claim \eqref{eq:claim_2},
we will conclude that $K\cap U=I$ 
which will prove the lemma.

We first prove \eqref{eq:claim_1},
fixing $k\in \{4,\dots,g\}$. 
Since the maps $\iota$ and $\kappa$ are injective,
the quotient modules $\Gr_k I=F_kI/F_{k-1} I$ and 
$\Gr_k (K \cap U) = F_k(K\cap U)/F_{k-1} (K \cap U)$
can be viewed as submodules of $\Gr_k U = F_k U/F_{k-1} U$.
Thus, claim \eqref{eq:claim_1} 
is equivalent to the injectivity of 
\begin{equation} \label{eq:Phi}
\Phi:= \Gr_k B^{(0)} : \frac{\Gr_k U}{\Gr_k I} \longrightarrow \Gr_k W.
\end{equation}
The quotient module  $\Gr_k U$
is generated by the  (classes of) elements 
$u(q,t,\ov{t})$, $u(q,r,t)$,   $u'(t;q,r)$, 
$u'(r;q,t)$ with $q,r,t \in S$ such that  
{$N(t)=k$, $N(q)<k$, $N(r)<k$}.
This induces a  generating system 
\begin{eqnarray}
\label{eq:gen_1} && \big\{\,  \overline{u}(q,t,\ov{t})
\ \big\vert \ q,t \in S, \hbox{ $N(t)=k$, $N(q)<k$} \, \big\}\\
\label{eq:gen_2}  &\cup &  \big\{\, \overline{u}(q,\ov{q},t)
\ \big\vert \ q,t \in S,  \hbox{ $N(t)=k$, $N(q)<k$} \, \big\}\\
\label{eq:gen_2'}  &\cup &  \big\{\, \overline{u}(q,r,t)
\ \big\vert \ q,r,t \in S, 
\hbox{ $N(t)=k$, $N(q)<k$, $N(r)<k$, $N(q)\neq N(r)$} \, \big\}\\
\label{eq:gen_3}  &\cup &  \big\{\,  \overline{u'}(t;q,r)
\ \big\vert \ q,r,t \in S, \hbox{ $N(t)=k$, $N(q)<k$, $N(r)<k$,
$N(q)\neq N(r)$} \, \big\}\\
\label{eq:gen_4}  &\cup & 
\big\{\,  \overline{u'}(r;q,t) 
\ \big\vert \ q,r,t \in S, \hbox{ $N(t)=k$, $N(q)<k$, $N(r)<k$,
$N(q)\neq N(r)$}
\, \big\}
\end{eqnarray}
of $\Gr_k U/\Gr_k I$, which we reduce as follows:
\begin{itemize}
\item For any $t,q,r,s\in S$ such that 
$N(t)=k$, $N(q)<k$, $N(r)<k$, $N(s)<k$ 
and $N(q),N(r), N(s)$ are pairwise different,  
the element
$$
\Diff{q}{\ov{q}}{t}{s}{r}{\ov{r}}
=\ve(t)\, u'(t;q,r) - \ve(s)\, u'(s;q,r) 
$$
is of type (ii), which shows that  $\overline{u'}(t;q,r)=0$.
Thus, all the generators of $\Gr_k U/\Gr_k I$ of type \eqref{eq:gen_3} 
are actually zero.
\item For any $t,q,r\in S$ such that 
$N(t)=k$, $N(q)<k$, $N(r)<k$ and $N(q)\neq N(r)$,  
the element
$$
\qquad 
\IHXthree{t}{q}{r} =
\ve(r)\, u'(q;t,r)  + \ve(q)\, u(r,\ov{r},t)
 - \ve(r)\, u'(t;q,r)  - \ve(t)\, u(r,\ov{r},q)
$$
is of type (iii), which (using the previous item) implies that
\begin{equation} \label{eq:uu}
\overline{u'}(q;t,r)
= -\ve(q)\ve(r) \, \overline{u}(r,\ov{r},t).
\end{equation}
Thus, all the generators of type \eqref{eq:gen_4}
are repeated in \eqref{eq:gen_2}, so that they can be removed.
\item For any $t,p,q\in S$ such that $N(t)=k$, $N(p)<k$, $N(q)<k$
and $N(p)\neq N(q)$, the element
$$
\IHXthree{p}{q}{t}
=\ve(t)\,u'(q;p,t) + \ve(q)\, u(t,\ov{t},p)
-\ve(t)\, u'(p;q,t)- \ve(p)\, u(t,\ov{t},q)
$$
is of type (iii),  which (using \eqref{eq:uu}), implies that
$$
 \overline{u}(t,\ov{t},p)=
\ve(t)\ve(p)\, \overline{u}(p,\ov{p},t) 
-\ve(t)\ve(p)\,  \overline{u}(q,\ov{q},t) 
+ \ve(p)\ve(q)\, \overline{u}(t,\ov{t},q).
$$
Thus, setting $q:=a_1$ and $t:=a_k$ in the above equation,
we see that
the submodule of $\Gr_k U/\Gr_k I$  generated by
\eqref{eq:gen_1} and \eqref{eq:gen_2}
is actually generated by the single element $u(a_1,a_k,b_k)$ 
of \eqref{eq:gen_1}  and \eqref{eq:gen_2}.
\item For any $t,q,r,s\in S$ such that 
$N(t)=k$, $N(q)<k$, $N(r)<k$, $N(s)<k$ 
and $N(q),N(r),N(s)$ are pairwise different, the element
\begin{eqnarray*}
\IHXthreep{t}{q}{r}{s}
&= &\ve(q)\, u(t,r,s)- \ve(t)\, u(q,r,s) -\ve(s)\,  u(\ov{r},\ov{t},{q})\\
&& - \ve(r)\, u(s,\ov{t},q) - \ve(r)\, u'(q;s,t) + \ve(s)\, u'(t;q,r)
\end{eqnarray*}
is of type (iv), which (using \eqref{eq:uu}) implies that
\begin{equation} \label{eq:uuu}
\ve(q)\ \overline{u}(t,r,s) +\ve(s)\, \overline{u}({r},{t},\ov{q})
+ \ve(r)\, \overline{u}(\ov{s},t,\ov{q})+
\ve(r)\ve(s)\ve(q)\, \overline{u}(s,\ov{s},t) =0.
\end{equation}
Thus, setting $(t,r):=(a_k,a_1)$ in the above equation, we get
$$
\overline{u}(\ov{s},\ov{q},a_k) =
 -\ve(s)\, \overline{u}(\ov{q},a_1,a_k) 
- \ve(q)\, \overline{u}(s,a_1,a_k) 
- \ve(s)\ve(q)\, \overline{u}(s,\ov{s},a_k),
$$
and, setting $(t,s):=(a_k,a_1)$ in the same equation, we get
$$
 \overline{u}(\ov{q},b_1,a_k)
= -\ve(q)\, \overline{u}(b_1,a_1,a_k) 
-  \ve(r)\, \overline{u}({r},\ov{q},a_k) 
-\ve(q)\ve(r)\, \overline{u}(r,a_1,a_k). 
$$
The above two equations show that, in the presence of \eqref{eq:gen_1} and \eqref{eq:gen_2},
the generators of type \eqref{eq:gen_2'} can be reduced
to those of the form $\overline{u}(q,a_1,a_k)$ with $q\in S$ such that $1<N(q)<k$.
Furthermore, by setting $(q,r,s,t):=(a_j,a_1,a_2,a_k)$ 
in \eqref{eq:uuu} for every $j\in \{3,\dots,k-1\}$, 
we  obtain
$$
\overline{u}(b_j,a_1,a_k)= 
-\overline{u}(a_2,b_2,a_k) - \overline{u}(b_2,a_k,b_j) 
- \overline{u}(a_2,a_1,a_k)
$$
and, 
by setting $(q,r,s,t):=(a_2,a_1,a_j,a_k)$ in \eqref{eq:uuu}, we also get
$$
\overline{u}(b_2,a_k,b_j) = -\overline{u}(a_j,a_1,a_k) - \overline{u}(b_2,a_1,a_k) - \, \overline{u}(a_j,b_j,a_k).
$$
We deduce that, 
in the presence of \eqref{eq:gen_1} and \eqref{eq:gen_2},
the generators of type \eqref{eq:gen_2'}  can be reduced
to those of the form $\overline{u}(b_2,a_1,a_k)$ and
$\overline{u}(a_j,a_1,a_k)$ 
with $j\in \{2,\dots,k-1\}$.
\end{itemize}
It follows from the above discussion
that $\Gr_k U/\Gr_k I$ is generated by the following elements:
$$
\overline{u}(a_1,a_k,b_k), \quad 
\overline{u}(a_i,b_i,a_k) \hbox{ with $i<k$}, \quad
\overline{u}(b_2,a_1,a_k), \quad 
\overline{u}(a_j,a_1,a_k) \hbox{ with $1<j<k$}.
$$
For $p,q\in S$ with $N(p)\leq k, N(q)\leq k$,
let $\overline{w}(p,q)$ be the class
of $w(p,q)$ in $\Gr_k W$.
According to~\eqref{eq:val_1}, the values of 
the map \eqref{eq:Phi} on this generating set 
are as follows:
\begin{eqnarray*}
\Phi\big( \overline{u}(a_1,a_k,b_k)\big)
&=& \overline{w}(a_k,b_k) + \overline{w}(a_1,b_k) 
- \overline{w}(a_1,a_k),\\
\Phi\big( \overline{u}(a_i,b_i,a_k) \big) 
&=& \overline{w}(b_i,a_k) -\overline{w}(a_i,a_k),  \\
\Phi\big( \overline{u}(b_2,a_1,a_k) \big) &=&
-\overline{w}(a_1,a_k) + \overline{w}(b_2,a_k),\\
\Phi\big( \overline{u}(a_j,a_1,a_k)  \big) &=& 
\overline{w}(a_1,a_k) + \overline{w}(a_j,a_k).
\end{eqnarray*}
Using the fact that 
$\big\{\overline{w}(a_i,a_k)\,\vert\, i<k\big\}
\cup \big\{\overline{w}(b_i,a_k)\,\vert\, i<k\big\} 
\cup \big\{\overline{w}(a_k,b_k)\big\}$ 
is a basis of $\Gr_k W$,
we easily check that
the above values of $\Phi$ are $\Q$-linearly independent.
Claim \eqref{eq:claim_1} follows.

We now prove \eqref{eq:claim_2}.
The quotient $Q:=F_3 U /F_3 I$
is generated by the classes 
\begin{equation} \label{eq:u_Q}
\overline{u}(p,q,r) \quad 
 \hbox{with $p,q,r \in F_3 S$ pairwise different}
\end{equation}
and the classes
\begin{equation} \label{eq:u'_Q}
 \overline{u'}(r;p,q)
\quad \hbox{with $p,q,r \in F_3 S$ 
such that $N(p),N(q),N(r)$ are pairwise different}
\end{equation}
of elements in $F_3U$
of type \eqref{eq:uuuu} and \eqref{eq:uuuu'}, respectively. 
For any $p,q,r \in F_3S$ such that 
$N(p),N(q),N(r)$ are pairwise different,
the element
$$
\Diff{p}{\ov{p}}{q}{\ov{q}}{r}{\ov{r}}
= \ve(q)\, u'(q;p,r) + \ve(q)\, u'(\ov{q};p,r)
$$
is of type (i), so that we have 
$\overline{u'}(q;p,r) = - \overline{u'}(\ov{q};p,r) \in Q$.
Thus, among the generators of $Q$ of type \eqref{eq:u'_Q},
we can reduce ourselves to 
$\overline{u'}(a_1;a_2,a_3)$, 
$\overline{u'}(a_2;a_1,a_3)$, 
$\overline{u'}(a_3;a_1,a_2)$.
Furthermore, for any $p,q,r \in F_3S$ such that 
$N(p),N(q),N(r)$ are pairwise different,
the element
$$
\IHXthree{p}{q}{r}
=\ve(r)\,u'(q;p,r) + \ve(q)\, u(r,\ov{r},p)
-\ve(r)\, u'(p;q,r)- \ve(p)\, u(r,\ov{r},q)
$$
is of type (iii), 
so that 
$\ve(q)\, \overline{u}(r,\ov{r},p)
= -\ve(r)\,\overline{u'}(q;p,r) 
+\ve(r)\, \overline{u'}(p;q,r)+ \ve(p)\, \overline{u}(r,\ov{r},q)\in Q$;
hence, among the generators of $Q$ 
of type \eqref{eq:u_Q}, we can remove
$\overline{u}(a_1,b_1,a_2)$, 
$\overline{u}(a_2,b_2,a_3)$ and
$\overline{u}(a_3,b_3,a_1)$.
Consequently, the module $Q$ is generated
by the following 10 elements:
\begin{eqnarray*}
\overline{u}(a_1,a_2,a_3), && \overline{u}(b_1,a_2,a_3), \ 
\overline{u}(a_1,b_2,a_3), \ 
\overline{u}(a_1,a_2,b_3), \\ 
\overline{u}(a_1,b_1,a_3), \
\overline{u}(a_2,b_2,a_1), \ 
\overline{u}(a_3,b_3,a_2), &&
\overline{u'}(a_1;a_2,a_3), \
\overline{u'}(a_2;a_1,a_3), \
\overline{u'}(a_3;a_1,a_2).    
\end{eqnarray*}
According to \eqref{eq:val_1}
and \eqref{eq:val_2}, 
the values of the map 
$\overline{B}:Q \to F_3 W \oplus \Q$
(which is induced by $B:F_3 U \to F_3 W \oplus \Q$)
on this generating set are as follows:
\begin{eqnarray*}
\overline{B}\big(\overline{u}(a_1,a_2,a_3) \big) &=&  
\big(w(a_2,a_3)+ w(a_1,a_3)+w(a_1,a_2) \,,\, -1/4 \big) \\
\overline{B}\big(\overline{u}(b_1,a_2,a_3)\big) &=& 
\big(-w(a_2,a_3)+w(a_1,b_3)+w(a_1,b_2) \,,\, +1/4 \big)  \\
\overline{B}\big(\overline{u}(a_1,b_2,a_3)\big) &=& 
\big(w(a_2,b_3)-w(a_1,a_3)+w(a_1,b_2) \,,\, +1/4 \big)  \\
\overline{B}\big(\overline{u}(a_1,a_2,b_3)\big) &=& 
\big(w(a_2,b_3)+w(a_1,b_3)-w(a_1,a_2) \,,\, +1/4\big)  \\
\overline{B}\big(\overline{u}(a_1,b_1,a_3)\big)  &=&
\big(w(a_1,b_3) -w(a_1,a_3)+w(a_1,b_1) \,,\, +1/4 \big)  \\
\overline{B}\big(\overline{u}(a_2,b_2,a_1)\big) &=&
\big(w(a_1,b_2)-w(a_1,a_2)+ w(a_2,b_2)\,,\, +1/4 \big)  \\
\overline{B}\big(\overline{u}(a_3,b_3,a_2)\big) &=& 
\big(w(a_2,b_3) -w(a_2,a_3)+ w(a_3,b_3) \,,\, +1/4 \big)  \\
\overline{B}\big(\overline{u'}(a_1;a_2,a_3)\big) &=&
\big(w(a_2,a_3)-w(a_2,b_3) \,,\, 0\big)  \\
\overline{B}\big(\overline{u'}(a_2;a_1,a_3)\big) &=&
\big(w(a_1,a_3)-w(a_1,b_3)  \,,\, 0\big)  \\
\overline{B}\big(\overline{u'}(a_3;a_1,a_2)\big) &=& 
\big(w(a_1,a_2)-w(a_1,b_2)  \,,\, 0\big)  
\end{eqnarray*}
Using that
$\big\{w(a_i,a_j)\,\vert\, 1 \leq i <j \leq 3 \big\}
\cup \big\{w(a_i,b_i)\,\vert\, 1 \leq i \leq 3 \big\}
\cup \big\{w(a_i,b_j)\,\vert\, 1 \leq i <j \leq 3 \big\}$
is a basis of the $\Q$-vector space $F_3W$,
it is easily checked that the above  vectors
of $F_3W\oplus \Q$ are $\Q$-linearly independent.
Therefore, the map $\overline{B}$ is injective.
\end{proof}

\subsection{Vanishing of $J$ on $K$} \label{subsec:vanishing}

We now prove that the map 
$J:\Lambda^2 (\Lambda^3 H) \to \Gamma_2\calI/\Gamma_3 \calI$ 
vanishes on the elements of type 
($\calR_i)$ with $0\leq i \leq 3$,
which have been identified in Theorem~\ref{th:K}.
Again, we use the action of the subgroup $G=G(S)$ of $\Sp(H)$
 defined at \eqref{eq:G}.
Since  the elements ($\calR_i)$ constitute a $G$-generating system of $K$ 
and since $J$ is $G$-equivariant,
this will prove that $J(K)=0$ and  conclude the proof of Theorem~B.

\subsubsection{BP-maps and PB-maps}

We start by fixing notations 
for some specific elements of the Torelli group $\calI$,
which we shall need in the sequel.

The first class of elements of $\calI$, 
which we have already  met in previous sections
and are called \emph{BP-maps} by Johnson,
are opposite Dehn twists  $T_{\varepsilon}T_\delta^{-1}$ 
along a pair $(\varepsilon,\delta)$ of disjoint and cobounding,
simple closed curves.
In fact, we shall only need such  ``B''ounding ``P''airs 
$(\varepsilon, \delta)$ that delimit 
a subsurface of genus~$1$ in $\Sigma$.
The data of a BP-map $T_{\varepsilon}T_\delta^{-1}$ of genus $1$
is encoded by 
giving a chain of $3$  circles in  $\Sigma$, whose neighborhood is a subsurface of genus $1$ with $2$ boundary components, 
as shown below:
\begin{equation} \label{eq:BP}
\begin{array}{c}
\includegraphics[scale=0.42]{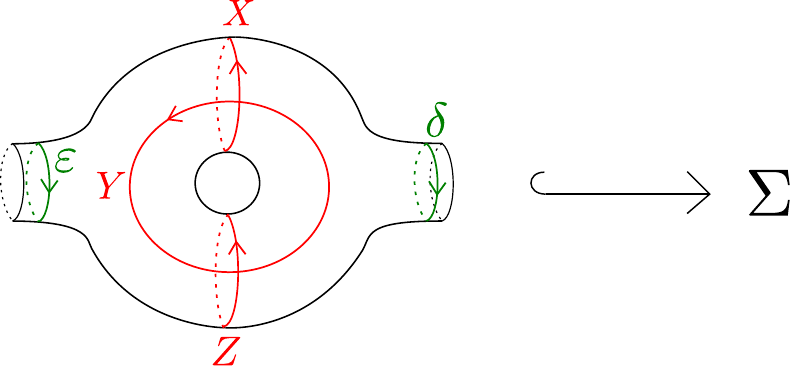}
\end{array},
\
\begin{array}{l}
\tau_1\big(T_{\varepsilon}T_\delta^{-1}\big)
= - x\wedge y \wedge e  =x \wedge y \wedge z\\
\hbox{where $x=\{X\},y=\{Y\},z=\{Z\}$, $e=\{\epsilon\}$.}
\end{array}
\end{equation}
(The value of $\tau_1$  on $T_{\varepsilon}T_\delta^{-1}$
is derived directly from \eqref{eq:tau1_beta}.)

The second class of elements of $\calI$ are called \emph{PB-maps},
since they arise from embeddings of the commutator subgroup 
of  the ``P''ure ``B''raid group   into $\calI$  \cite{Oda92}.
Specifically, a PB-map in~$\calI$ is encoded by a $Y$-shaped graph 
in the surface $\Sigma$,  
which gives a 2-holed disk in $\Sigma$
and defines a commutator of two Dehn twists 
$[T_{\varepsilon},T_\delta]$, as shown below:
\begin{equation} \label{eq:PB}
\begin{array}{c}
\includegraphics[scale=0.25]{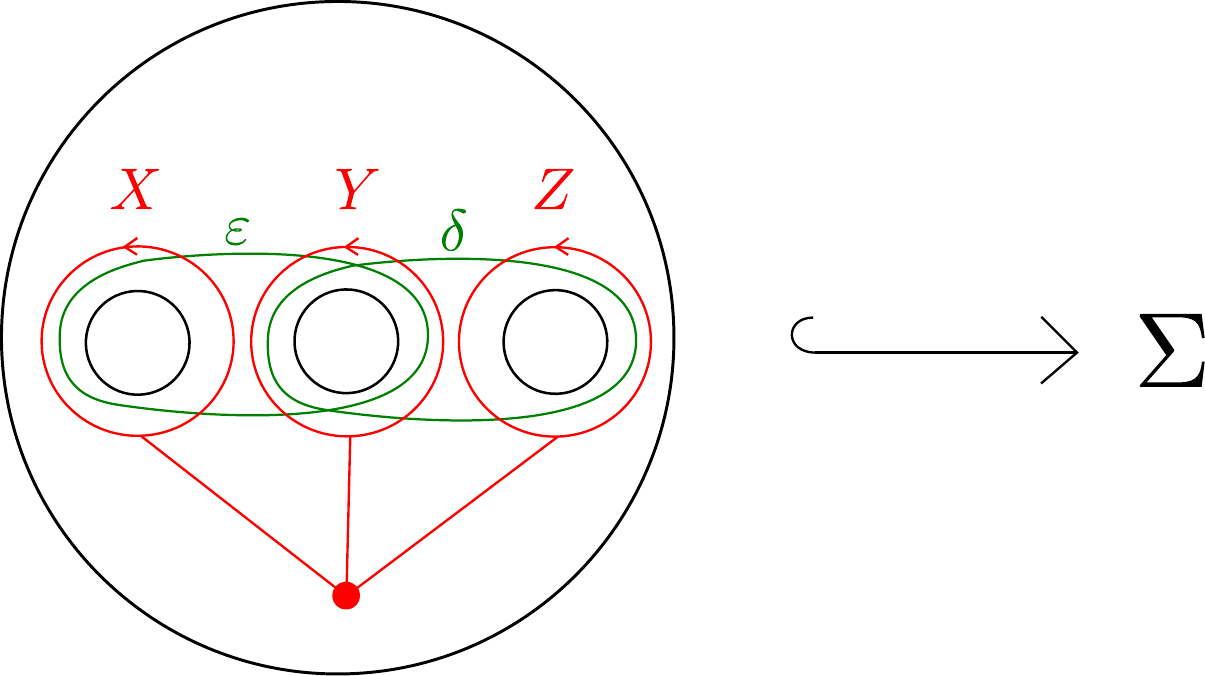}
\end{array},
\quad
\begin{array}{l}
\tau_1\big([T_{\varepsilon},T_\delta]\big)
 = - x \wedge y \wedge z\\
\hbox{where $x=\{X\},y=\{Y\},z=\{Z\}$.}
\end{array}
\end{equation}
(The value of $\tau_1$ on $[T_{\varepsilon},T_\delta]$
can be derived, for instance, from the exact 
relationship between Milnor invariants 
and Johnson homomorphisms \cite{GH02};
BP-maps also appear in \cite{Put09}
under the name ``commutators of simply intersecting pairs''.)

Recall that $N:S \to \{1,\dots,g\}$ is the map 
defined by $N(a_i)=N(b_i)=i$,
and observe that any elementary trivector $s_1\wedge s_2 \wedge s_3$
(with $s_1,s_2,s_3\in S$ pairwise different)
can be realized as the value of $\tau_1$ on
\begin{itemize}
\item  a PB-map
if $N(s_1),N(s_2),N(s_3)$ are pairwise different,
\item  a BP-map if $N(s_i)=N(s_j)$ for some $i\neq j$.
\end{itemize}

\begin{example}
The elementary trivectors $a_1 \wedge a_2 \wedge b_3$ 
and $a_1 \wedge b_1 \wedge b_3$
are (up to signs)
realized by the following PB-map and BP-map, respectively:
$$
\includegraphics[scale=0.48]{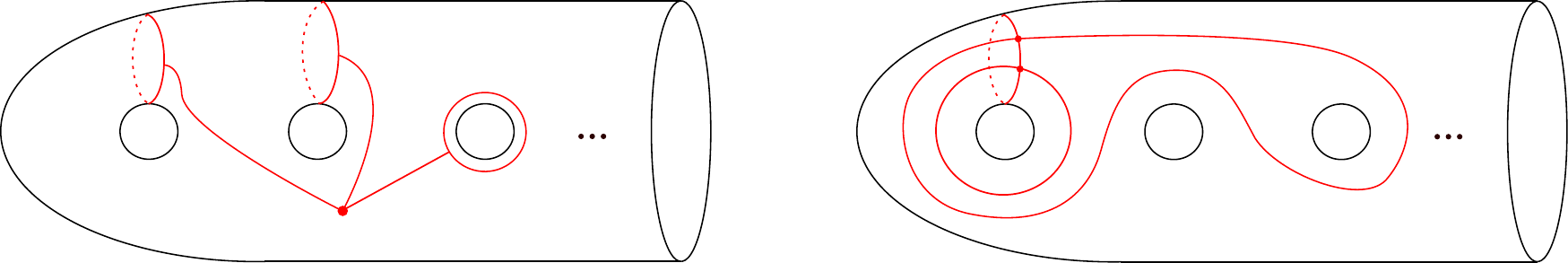}
$$

\hfill $\blacksquare$
\end{example}

\subsubsection{Vanishing of $J$ on the elements ($\calR_0)$}

All the elements of the family ($\calR_0)$
 are of the form
$\Com{s_1}{s_2}{s_3}{s'_1}{s'_2}{s'_3}$
where the elementary trivectors $s_1\wedge s_2 \wedge s_3$
and $s_1'\wedge s_2' \wedge s_3'$ are realized by
BP-maps or PB-maps with disjoint supports:
specifically, realizations are shown (up to a sign)
in Table \ref{table:no_contraction}.
Therefore, all the elements of ($\calR_0)$ belong to $\ker(J)$.

\begin{table}[h!]
\centering
 \begin{tabular}{||c c c ||} 
 \hline && \\
 $\begin{array}{c}
 \hbox{
 \includegraphics[scale = 0.3]{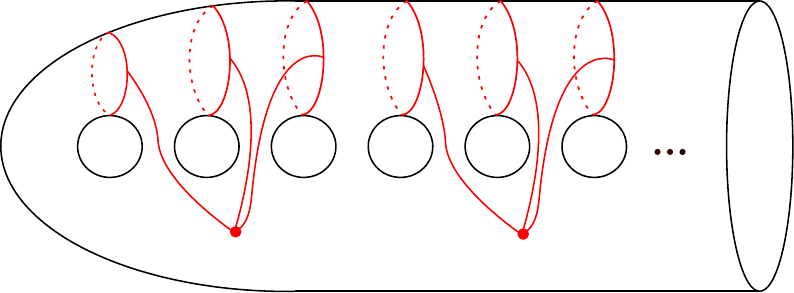}} \\ 
 \Com{a_1}{a_2}{a_3}{a_4}{a_5}{a_6} 
 \end{array}$ & 
  $\begin{array}{c}
 \hbox{\includegraphics[scale = 0.3]{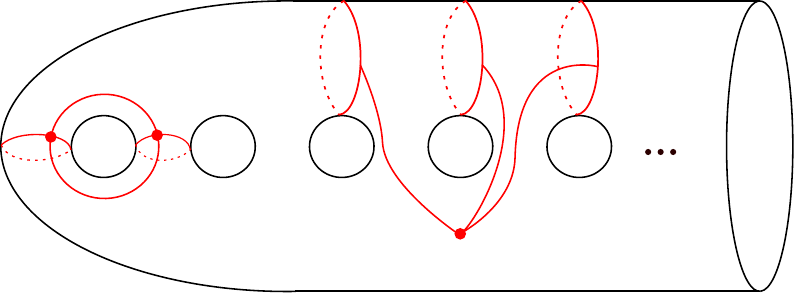}} \\ 
 \Com{a_1}{b_1}{a_2}{a_3}{a_4}{a_5}
  \end{array}$& 
  $\begin{array}{c}
 \hbox{\includegraphics[scale = 0.3]{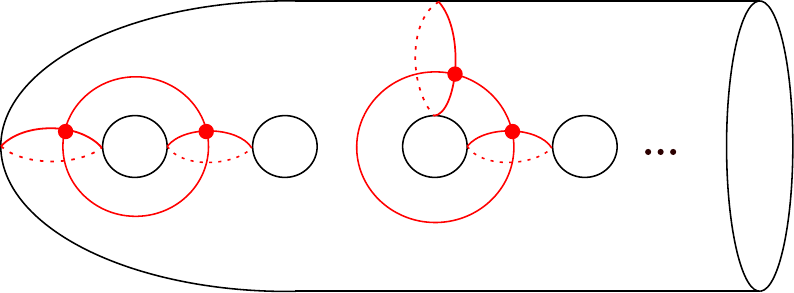}} \\ 
 \Com{a_1}{b_1}{a_2}{a_3}{b_3}{a_4}
  \end{array}$\\[0.5cm] 
  $\begin{array}{c}
 \hbox{\includegraphics[scale = 0.3]{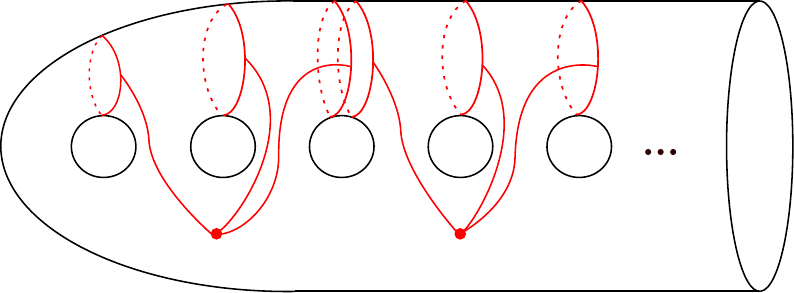}} \\ 
 \Com{a_1}{a_2}{a_3}{a_3}{a_4}{a_5}
  \end{array}$& 
  $\begin{array}{c}
 \hbox{\includegraphics[scale = 0.3]{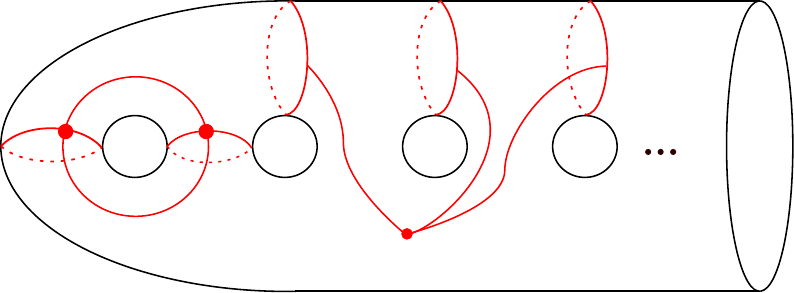}} \\ 
 \Com{a_1}{b_1}{a_2}{a_2}{a_3}{a_4}
  \end{array}$& 
  $\begin{array}{c}
 \hbox{\includegraphics[scale = 0.3]{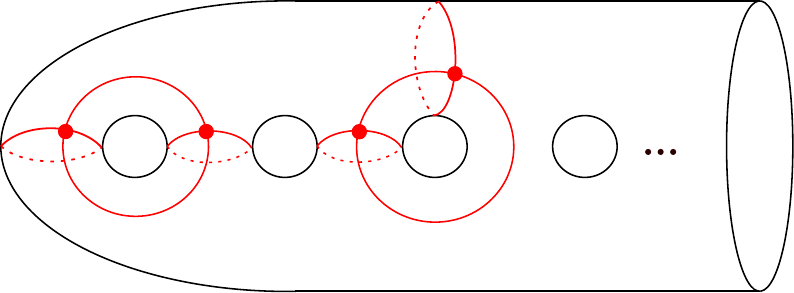}} \\ 
 \Com{a_1}{b_1}{a_2}{a_2}{a_3}{b_3}
  \end{array}$\\[0.5cm] 
 &  $\begin{array}{c}
 \hbox{\includegraphics[scale = 0.3]{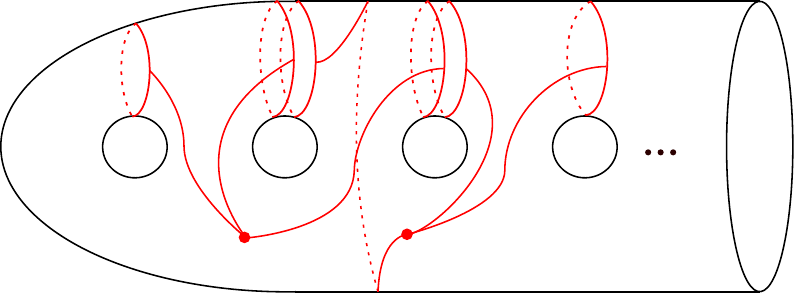}} \\ 
 \Com{a_1}{a_2}{a_3}{a_2}{a_3}{a_4}
  \end{array}$& \\ && \\
 \hline
 \end{tabular}
 \caption{Realizations of elementary trivectors
 with no contraction}
 \label{table:no_contraction}
\end{table}

Since $(\calR_0)$ is a $G$-generating system of $K\cap U_0$, 
we obtain the following more general result.

\begin{lemma} \label{lem:0}
The submodule $K\cap U_0$ of $\Lambda^2(\Lambda^3H)$
is contained in $\ker(J)$.
\end{lemma}

\subsubsection{Vanishing of $J$ on the elements $(\calR_1)$}

We deduce from the following lemma that
the first six elements of ($\calR_1)$ are in the kernel of $J$.

\begin{lemma} \label{lem:D}
Let $x,y,c_1,c_2,u,v \in S$  be pairwise different,
and assume that $c_1\neq \ov{c_2}$ too.
Then we have \
$\Diff{x}{y}{c_1}{c_2}{\ov{u}}{\ov{v}} \in \ker(J).$
\end{lemma} 

\begin{proof}
Let $\eta := \varepsilon(c_1)\, \varepsilon(c_2)$.
We claim that 
$$
C:= \Com{x}{y}{c_1+c_2}{-\eta\, \ov{c_1} + \ov{c_2}}{\ov{u}}{\ov{v}} 
\in \ker(J).
$$
Since $\Com{x}{y}{c_1}{\ov{c_2}}{\ov{u}}{\ov{v}}$
and  $\Com{x}{y}{c_2}{\ov{c_1}}{\ov{u}}{\ov{v}}$
belong to $\ker(J)$ by Lemma \ref{lem:0},
we deduce that 
$-\eta\,\Com{x}{y}{c_1}{\ov{c_1}}{\ov{u}}{\ov{v}}+
\Com{x}{y}{c_2}{\ov{c_2}}{\ov{u}}{\ov{v}}
= -\varepsilon(c_2)\, \Diff{x}{y}{c_1}{c_2}{\ov{u}}{\ov{v}}$ does too.

We now prove the above claim. If 
${c_2} \notin \{\ov{x},\ov{y}\}$ and $c_1\not\in \{\ov{u},\ov{v}\}$,
then we consider the symplectic transformation $T$
that maps $(c_1,\ov{c_1},c_2,\ov{c_2})$
to  $(c_1+c_2,\ov{c_1},c_2,-\eta\, \ov{c_1}+\ov{c_2})$
and fixes all other elements of $S$; then $T$ maps
$$
\Com{x}{y}{c_1}{\ov{c_2}}{\ov{u}}{\ov{v}} \in \ker(J)
$$
to $C$, so that $C\in \ker(J)$.
If we have ${c_1} \notin \{\ov{x},\ov{y}\}$ 
and $c_2\not\in \{\ov{u},\ov{v}\}$, we can proceed similarly
by exchanging the roles of $c_1$ and $c_2$.

We now assume that none of the above two conditions occur,
i.e$.$ we have $\{c_1,c_2\}=\{\ov{x},\ov{y}\}$
or (exclusively) $\{c_1,c_2\}=\{\ov{u},\ov{v}\}$. For instance,
we assume that $(c_1,c_2)=(\ov{x},\ov{y})$,
the other cases being treated similarly. 
Then 
$$
C=
\Com{\ov{c_1}}{-\eta \ov{c_1}+\ov{c_2}}{c_1+c_2}{-\eta \ov{c_1}+\ov{c_2}}{\ov{u}}{\ov{v}}
= T\cdot  \Com{\ov{c_1}}{\ov{c_2}}{c_1}{\ov{c_2}}{\ov{u}}{\ov{v}},
$$
and we deduce from Lemma \ref{lem:0}
that $C$ belongs to $\ker(J)$ in the present case too.
\end{proof}

Besides, we have the following
variant of Lemma~\ref{lem:D}
(which will be used later).

\begin{lemma} \label{lem:D_bis}
Let $x,y,c,u,v \in S$  be pairwise different,
and assume that $c\not\in \{\ov{x},\ov{y},\ov{u},\ov{v}\}$.
Then we have 
$\Diff{x}{y}{c}{\ov{c}}{\ov{u}}{\ov{v}} \in \ker(J).$
\end{lemma} 

\begin{proof}
Let $T$ be the symplectic transformation that maps $c$ to $c+\ov{c}$
and fixes all other elements of $S$. We have
$$
T \cdot \Com{x}{y}{c}{c}{\ov{u}}{\ov{v}} -\Com{x}{y}{c}{c}{\ov{u}}{\ov{v}} 
-\Com{x}{y}{\ov{c}}{\ov{c}}{\ov{u}}{\ov{v}} 
=  \ve(c)\, \Diff{x}{y}{c}{\ov{c}}{\ov{u}}{\ov{v}} 
$$
and we conclude with Lemma \ref{lem:0}.
\end{proof}

The next lemma deals 
with the seventh (and last) element of $(\calR_1)$.

\begin{lemma}[Gervais \& Habegger]
We have \
$
\IHXone{a_{1}}{a_{2}}{a_{3}}{a_{4}}{b_{1}} \in \ker (J).
$
\end{lemma} 

\begin{proof}
This has been proved by Gervais and Habegger \cite{GH02} but,
for the sake of completeness, we repeat their arguments
(with some variations).
We focus on a subsurface of $\Sigma$ of genus $4$
which contains the based loops 
$\alpha_1,\beta_1,\dots,\alpha_4,\beta_4$ of Figure \ref{fig:basis}.
This subsurface (union a rectangle) is shown in Figure \ref{fig:IHX}
as the cylinder $D\times [0,1]$ 
over a 4-holed disk $D$: 
the curves $X,Y,Z$ (shown in Figure~\ref{fig:IHX})
lie on $D \times \{1\}$
and we denote by $X_0,Y_0,Z_0$
the corresponding  curves  on $D \times \{0\}$
(not shown in Figure~\ref{fig:IHX}).
\begin{figure}[h!]
\centering
\includegraphics[scale=0.4]{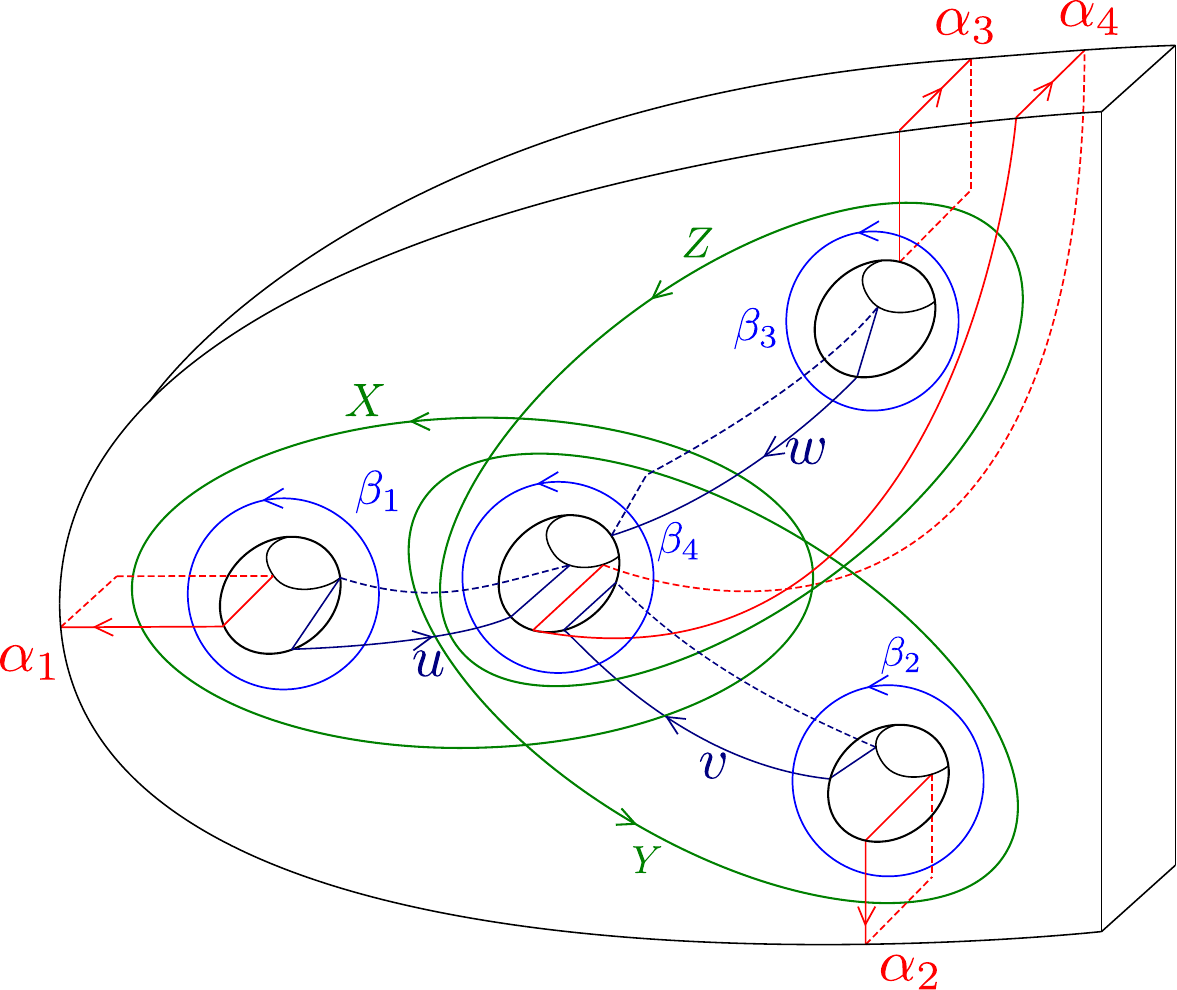}
\caption{The curves $X,Y,Z$ and $U,V,W$
on $\partial(D\times [0,1])$}
\label{fig:IHX}
\end{figure}

In the rest of this proof, the Dehn twist $T_C$  
along a simple closed curve $C$ is simply denoted by $C$
and, if $C$ is oriented, 
the homology class $\{C\}\in H$ is denoted
by the lower case roman letter~$c$. 
By the Hall--Witt identity, we know that
$$
\big[{}^ZX,[Y,Z]\big] \,
\big[{}^YZ,[X,Y]\big] \,
\big[{}^XY,[Z,X]\big] = 1 \ \in \calM.
$$
and, 
since the Dehn twists along  $X_0,Y_0,Z_0$ 
commute with those along $X,Y,Z$, we obtain that
$$
\big[{}^Z(XX_0^{-1}),[Y,Z]\big] \,
\big[{}^Y(ZZ_0^{-1}),[X,Y]\big] \,
\big[{}^X(YY_0)^{-1},[Z,X]\big] = 1 \ \in \calM.
$$
Note that $XX_0^{-1},YY_0^{-1},ZZ_0^{-1}$ are BP-maps (of genus $1$)
while $[Y,Z],[X,Y],[Z,X]$ are PB-maps:
hence the above identity gives a degree $2$ relation
in $\Gr^\Gamma \calI$.
(See \cite{Put09} for a systematic use
of Hall--Witt identities in presentations of the group $\calI$.)
Consequently, the sum
$$
\big(Z_* \cdot \tau_1\big(XX_0^{-1}\big)\big) \wedge \tau_1([Y,Z])+
\big(Y_* \cdot \tau_1\big(ZZ_0^{-1}\big)\big) \wedge \tau_1([X,Y])+
\big(X_* \cdot \tau_1\big(YY_0^{-1}\big)\big) \wedge \tau_1([Z,X])
$$
belongs to $\ker(J)$. 
(Here $X_*\in \Sp(H)$ denotes the action of the Dehn twist along $X$.)
Using \eqref{eq:BP} and \eqref{eq:PB}, 
we obtain
\begin{eqnarray*}
 Z_* \cdot \tau_1\big(XX_0^{-1}\big)
&=& \varepsilon\, b_1 \wedge Z_*(u) \wedge b_4 \\
& = & \varepsilon\, b_1 \wedge (u+z) \wedge b_4 \\
& = & \varepsilon\, b_1 \wedge ((a_1-a_4)+(b_3+b_4)) \wedge b_4 
\ = \ \varepsilon\, b_1 \wedge ( a_1-a_4 +b_3 ) \wedge b_4     
\end{eqnarray*}
and
$$
\tau_1([Y,Z]) = \eta\, b_2 \wedge b_4 \wedge b_3
$$
respectively.
(Here $\varepsilon$ and $\eta$ are fixed signs, 
which we do not need to determine.)
Next, it follows from Lemma \ref{lem:0} that
$$
\big(Z_* \cdot \tau_1\big(XX_0^{-1}\big)\big) \wedge \tau_1([Y,Z])
\equiv - \varepsilon \eta \,
\Com{b_1}{a_4}{b_4}{b_2}{b_4}{b_3}
\mod \ker(J);
$$
similarly, we obtain
$$
\big(Y_* \cdot \tau_1\big(ZZ_0^{-1}\big)\big) \wedge \tau_1([X,Y])
\equiv - \varepsilon \eta \,
\Com{b_3}{a_4}{b_4}{b_1}{b_4}{b_2}
\mod \ker(J)
$$
and 
$$
\big(X_* \cdot \tau_1\big(YY_0^{-1}\big)\big) \wedge \tau_1([Z,X])
\equiv - \varepsilon \eta \,
\Com{b_2}{a_4}{b_4}{b_3}{b_4}{b_1}
\mod \ker(J).
$$
Therefore the element
$$
\underbrace{\Com{b_4}{b_1}{a_4}{b_4}{b_3}{b_2}
+ \Com{b_4}{b_3}{a_4}{b_4}{b_2}{b_1}
+ \Com{b_4}{b_2}{a_4}{b_4}{b_1}{b_3}}_{\IHXone{b_4}{b_1}{b_3}{b_2}{a_4}}
$$
belongs to $\ker(J)$, and we conclude using the $G$-action.
\end{proof}

Since $(\calR_1)$ is a $G$-generating system of $K\cap U_1$, 
we obtain the following.

\begin{lemma} \label{lem:1}
The submodule $K\cap U_1$ of $\Lambda^2(\Lambda^3H)$
is contained in $\ker(J)$.
\end{lemma}

\subsubsection{Vanishing of $J$ on the elements  ($\calR_2)$.}

We now deal with the elements of the family ($\calR_2)$
through the next three lemmas.

\begin{lemma} \label{lem:IHX2}
The elements 
$\IHXtwo{a_1}{a_2}{a_3}{b_1}$  and $\IHXtwo{a_1}{a_2}{a_3}{a_4}$
belong to $\ker(J)$.
\end{lemma}

\begin{proof}
For all $h\in \{3,\dots, g\}$, the element
$
\sum_{i=1}^{h-1} \Com{a_1}{a_2}{b_2}{a_h}{a_i}{b_i}
$
of $\Lambda^2(\Lambda^3 H)$
belongs to $\ker(J)$, since ${a_1}\wedge{a_2}\wedge{b_2}$
and $\sum_{i=1}^{h-1} {a_h}\wedge{a_i}\wedge{b_i}$ 
can be realized by BP maps (of genus $1$ and $h-1$, respectively)
whose supports are disjoint one from the other.
Setting $h:=3$, we obtain that
\begin{equation} \label{eq:W}
W:= \Com{a_1}{a_2}{b_2}{a_3}{a_1}{b_1}
+  \Com{a_1}{a_2}{b_2}{a_3}{a_2}{b_2}    
\end{equation}
belongs to $\ker(J)$.
The element of $G \simeq \Z_4^g \rtimes   \mathfrak{S}_g$
corresponding to the transposition $(23)$
maps $W$ to $\IHXtwo{a_1}{a_2}{a_3}{b_1}$ 
which, therefore, is  in $\ker(J)$ too.

Next, the element of $G$ corresponding to the cycle $(1423)$
transforms $W$ to 
$$ 
E := \Com{a_4}{a_3}{b_3}{a_1}{a_3}{b_3} + \Com{a_4}{a_3}{b_3}{a_1}{a_4}{b_4}
$$
which, therefore, belongs to $\ker(J)$.
Let also $T$ be the symplectic transformation
that maps $(a_4,b_2)$ to  $( a_4 + a_2 ,  b_2 - b_4)$ 
and leaves all other elements of $S$ fixed.
 We have 
 \begin{eqnarray*}
 {T} \cdot E - E   & \equiv & \Com{a_2}{a_3}{b_3}{a_1}{a_3}{b_3}
 + \Com{a_4}{a_3}{b_3}{a_1}{a_2}{b_4}  \mod K \cap U_0.
 \end{eqnarray*}
Hence, using Lemma \ref{lem:0} and the $G$-action, we obtain that 
$$
\IHXtwo{a_1}{a_2}{a_3}{a_4} = 
\Com{a_1}{a_3}{b_3}{a_2}{a_3}{b_3} 
- \Com{a_1}{a_2}{a_4}{a_3}{b_3}{b_4}
\ \in \ker(J).
$$

\up
\end{proof}

\begin{lemma}
The following five elements belong to $\ker(J)$:
$$
\hbox{\footnotesize 
$\Square{a_1}{a_2}{a_2}{b_1}{a_3}{b_3}$,
$\Square{a_1}{a_2}{a_2}{a_4}{a_3}{b_3}$,
$\Square{a_1}{a_2}{a_4}{a_5}{a_3}{b_3}$,
$\Square{a_1}{a_2}{b_1}{a_4}{a_3}{b_3}$,
$\Square{a_1}{a_2}{a_4}{b_4}{a_3}{b_3}$.}
$$
\end{lemma}

\begin{proof}
Consider the element $W\in \ker(J)$ defined by \eqref{eq:W},
and let $T_0$ be the symplectic transformation that maps
$(a_2,a_3)$ to $(a_2+b_3,a_3+b_2)$. We have
$$
T_0\cdot W-W = \Square{a_1}{b_2}{b_2}{b_1}{b_3}{a_3}
+ \Diff{b_3}{b_2}{a_1}{a_3}{b_2}{a_1}+ W'
$$
where 
$W':=\Com{a_1}{a_2}{b_2}{a_3}{b_3}{b_2}
-\Com{a_1}{b_3}{a_3}{b_2}{a_3}{b_3}$.
But, an appropriate element of $G$ transforms $W'$ to $W$,
and we deduce from Lemma \ref{lem:D}
that  $\Square{a_1}{b_2}{b_2}{b_1}{b_3}{a_3}$ 
belongs to $\ker(J)$.
Using again the $G$-action, we conclude that 
$\Square{a_1}{a_2}{a_2}{b_1}{a_3}{b_3}$ is in $\ker(J)$ too.

By Lemma \ref{lem:D}, 
$D:=\Diff{a_1}{a_3}{a_2}{a_4}{a_2}{a_4}$ belongs
to $\ker(J)$.
For the symplectic transformation $T_1$
that maps $(a_3,a_4)$ to $(a_3+b_4,a_4+b_3)$ 
and fixes all other elements of $S$,
we have
\begin{eqnarray*}
T_1\cdot D -D & \equiv & D_1 \mod (K\cap U_0 + K\cap U_1)    
\end{eqnarray*}
where 
\begin{eqnarray*}
D_1 &:=&
\Com{a_1}{a_3}{a_2}{b_2}{a_2}{b_3} + \Com{a_1}{b_4}{a_2}{b_2}{a_2}{a_4} \\
&& \quad - \Com{a_1}{a_3}{a_4}{b_4}{a_2}{b_3} - \Com{a_1}{b_4}{a_4}{b_4}{a_2}{a_4}
\ =  \ -\Square{a_1}{a_2}{a_2}{a_3}{b_4}{a_4}.
\end{eqnarray*}
Hence, by Lemma \ref{lem:0} and Lemma \ref{lem:1},
$D_1$ belongs to $\ker(J)$.
Using the $G$-action, we deduce that
$\Square{a_1}{a_2}{a_2}{a_4}{a_3}{b_3}$ does too.


Proceeding with the same strategy, we consider 
the symplectic transformation $T_3$
mapping  $(a_2,b_5)$ to $(a_2+a_5,b_5-b_2)$
and fixing all other elements of $S$.
We have
\begin{eqnarray*}
&& T_3\cdot \Square{a_1}{a_2}{a_2}{a_4}{a_3}{b_3}
-  \Square{a_1}{a_2}{a_2}{a_4}{a_3}{b_3} \\
& = & 
D_3 + \Diff{a_1}{a_5}{a_3}{a_4}{a_2}{b_2}
+ \Diff{a_1}{a_5}{a_3}{a_4}{a_5}{b_2}
\end{eqnarray*}
where 
\begin{eqnarray*}
D_3 &:=& - \Com{a_1}{a_2}{a_4}{a_5}{b_2}{b_4}
-\Com{a_1}{a_4}{b_3}{a_5}{b_4}{a_3} 
\\
&& \quad + \Com{a_1}{a_3}{b_3}{a_5}{b_3}{a_3}
+ \Com{a_1}{a_2}{a_3}{a_5}{b_2}{b_3} 
\ = \ \Square{a_1}{a_5}{a_2}{a_4}{a_3}{b_3}.
\end{eqnarray*}
Therefore, using the $G$-action, we deduce that
$\Square{a_1}{a_2}{a_4}{a_5}{a_3}{b_3}$ is in $\ker(J)$.

We continue with  the symplectic transformation $T_4$
mapping  $(a_2,b_4)$ to $(a_2+a_4,b_4-b_2)$
and fixing all other elements of $S$.
We have
\begin{eqnarray*}
&& T_4\cdot \Square{a_1}{a_2}{a_2}{b_1}{a_3}{b_3}
-  \Square{a_1}{a_2}{a_2}{b_1}{a_3}{b_3} \\
& = &  D_4
- \Diff{a_1}{a_4}{a_3}{b_1}{a_2}{b_2}
- \Diff{a_1}{a_4}{a_3}{b_1}{a_4}{b_2}
\end{eqnarray*}
where 
\begin{eqnarray*}
D_4 &:=& - \Com{a_1}{a_2}{b_1}{a_4}{b_2}{a_1}
-\Com{a_1}{b_1}{b_3}{a_4}{a_1}{a_3}
\\
&& \quad - \Com{a_1}{a_3}{b_3}{a_4}{b_3}{a_3}
- \Com{a_1}{a_2}{a_3}{a_4}{b_2}{b_3}
\ = \ \Square{a_1}{a_4}{a_2}{b_1}{a_3}{b_3}.
\end{eqnarray*}
Therefore, using the $G$-action, we deduce that
$\Square{a_1}{a_2}{b_1}{a_4}{a_3}{b_3}$ is in $\ker(J)$.

Finally, using again the above  transformation $T_1$,
 an elementary computation shows that the difference
$ T_1\cdot \Com{a_1}{a_3}{a_4}{a_2}{a_3}{a_4} 
- \Com{a_1}{a_3}{a_4}{a_2}{a_3}{a_4}$ decomposes as
$$
\Com{a_1}{b_4}{b_3}{a_2}{a_3}{a_4}
+\Com{a_1}{b_4}{a_4}{a_2}{b_4}{a_4}
+\Com{a_1}{a_3}{b_3}{a_2}{a_3}{b_3} 
+\Com{a_1}{a_3}{a_4}{a_2}{b_4}{b_3}
$$
plus a sum of 11 terms  which belongs to $K\cap U_0 + K\cap U_1$.
Using Lemma \ref{lem:0} and Lemma \ref{lem:1},
we conclude that
$\Square{a_1}{a_2}{a_4}{b_4}{a_3}{b_3}$ belongs to $\ker(J)$.
\end{proof}

\begin{lemma}
We have $\Triple{a_1}{a_2}{a_3}{a_4} \in \ker(J)$.
\end{lemma}

\begin{proof}
Let $T$ be the symplectic transformation
that maps $(a_2,b_1)$ to 
$(a_2 + a_1, b_1 - b_2)$ and leaves all other elements of $S$ fixed.
We have 
\begin{eqnarray*}
&&T \cdot \Square{a_1}{a_2}{a_2}{a_4}{a_3}{b_3}  -  \Square{a_1}{a_2}{a_2}{a_4}{a_3}{b_3}\\
&=& \Com{a_1}{a_2}{a_3}{a_1}{b_2}{b_3}
- \Com{a_1}{a_2}{a_4}{a_1}{b_2}{b_4}
- \Com{a_1}{a_4}{b_3}{a_1}{b_4}{a_3},
\end{eqnarray*}
which is $-\Triple{a_1}{a_2}{a_4}{a_3}$.
Using the $G$-action, 
we obtain that 
$\Triple{a_1}{a_2}{a_3}{a_4} \in \ker(J)$.
\end{proof}

Since $(\calR_2)$ is a $G$-generating system of $K\cap U_2$, 
we obtain the following.

\begin{lemma} \label{lem:2}
The submodule $K\cap U_2$ of $\Lambda^2(\Lambda^3H)$
is contained in $\ker(J)$.
\end{lemma}

\subsubsection{Vanishing of $J$ on the elements  ($\calR_3)$.}

We deduce from Lemma~\ref{lem:D} and Lemma~\ref{lem:D_bis}
that the first two elements of ($\calR_3)$
belong to $\ker(J)$.
The following  deals with the
last two elements of ($\calR_3)$,
and finishes the proof of Theorem~B.

\begin{lemma}
The elements $\IHXthree{a_1}{a_2}{a_3}$ and
  $\IHXthreep{a_1}{a_2}{a_3}{a_4}$ belong to $\ker(J)$.
\end{lemma}

\begin{proof}
Let $T_1$ be the symplectic transformation that maps
$(a_1,a_2)$ to $( a_1 + b_2,a_2 +b_1)$, and set 
\begin{eqnarray*}
E_1 &:=& \Com{b_2}{a_3}{b_3}{b_1}{a_3}{b_3}
+ \Com{b_2}{a_2}{b_1}{a_3}{b_3}{b_2},\\
F_1 &:=& \Com{a_1}{a_3}{b_3}{b_1}{a_3}{b_3} 
+ \Com{b_2}{a_3}{b_3}{a_2}{a_3}{b_3} \\
&&+ \ \Com{a_1}{a_3}{b_3}{b_1}{a_2}{b_2} + \Com{b_2}{a_3}{b_3}{a_1}{b_1}{a_2}.
\end{eqnarray*}
We have 
\begin{eqnarray*}
 T_1 \cdot \IHXtwo{a_1}{a_2}{a_3}{b_1}  - \IHXtwo{a_1}{a_2}{a_3}{b_1}
&= &  E_1 + F_1 .
\end{eqnarray*}
Since $E_1$ belongs to $K \cap(V_{1,1}+V_2)=K \cap U_2$,
it follows from   Lemma \ref{lem:IHX2} and Lemma \ref{lem:2}
that $F_1$ belongs to $\ker(J)$.
Besides, we have 
$$
\IHXthree{a_1}{a_2}{a_3} 
= - F_1 + 
\underbrace{\Com{a_3}{b_3}{a_1}{b_1}{a_2}{b_2} 
+ \Com{a_3}{b_3}{b_1}{a_1}{a_2}{b_2}}_{
\mathsf{D}\left({a_3},{b_3}\,;\,{a_1}
,{b_1}\,;\,{a_2},{b_2}\right)
}.
$$
Thus, an application of Lemma \ref{lem:D_bis}
shows that $J$ vanishes on $\IHXthree{a_1}{a_2}{a_3}$.

Finally, consider the symplectic transformation $T_2$
that maps $(a_2,a_3)$ to $(a_2 + b_3,a_3 +b_2)$
and fixes all other elements of $S$.
We have 
$$
T_2\cdot \Square{a_3}{a_2}{a_2}{a_4}{a_1}{b_1} 
- \Square{a_3}{a_2}{a_2}{a_4}{a_1}{b_1} 
\equiv E_2 \mod (K\cap V_1+K\cap V_2)
$$
where
\begin{eqnarray*}
E_2& := & 
-\Com{b_2}{a_2}{a_4}{a_2}{b_2}{b_4} 
-\Com{a_3}{b_3}{a_4}{a_2}{b_2}{b_4}
-\Com{a_3}{a_2}{a_4}{b_3}{b_2}{b_4}\\
&& + \Com{b_2}{a_2}{a_1}{a_2}{b_2}{b_1} 
+ \Com{a_3}{b_3}{a_1}{a_2}{b_2}{b_1}
+ \Com{a_3}{a_2}{a_1}{b_3}{b_2}{b_1}\\
&& -\Com{b_2}{a_4}{b_1}{a_2}{b_4}{a_1} - \Com{a_3}{a_4}{b_1}{b_3}{b_4}{a_1}\\
&& +\Com{b_2}{a_1}{b_1}{a_2}{b_1}{a_1} + \Com{a_3}{a_1}{b_1}{b_3}{b_1}{a_1}.
\end{eqnarray*}
It follows from Lemma \ref{lem:1} and Lemma \ref{lem:2}
that $E_2$ is in $\ker(J)$.
Next, we observe that
\begin{eqnarray*}
E_2
&=&  \IHXthreep{a_1}{a_4}{a_2}{a_3} 
+ \Com{a_3}{b_3}{a_4}{b_4}{a_1}{b_1} - \Com{a_4}{b_4}{a_1}{b_1}{a_2}{b_2}\\
&&  -\Com{b_2}{a_2}{a_4}{a_2}{b_2}{b_4} 
+ \Com{b_2}{a_2}{a_1}{a_2}{b_2}{b_1} \\
&& +\Com{b_2}{a_1}{b_1}{a_2}{b_1}{a_1} 
+ \Com{a_3}{a_1}{b_1}{b_3}{b_1}{a_1} \\
&& -\Com{a_3}{b_3}{a_4}{a_2}{b_2}{b_4}  
+ \Com{a_3}{b_3}{a_1}{a_2}{b_2}{b_1} \\
&=&  \IHXthreep{a_1}{a_4}{a_2}{a_3} 
+ \Com{a_3}{b_3}{a_4}{b_4}{a_1}{b_1} - \Com{a_4}{b_4}{a_1}{b_1}{a_2}{b_2}\\
&& +  \IHXthree{a_1}{a_4}{a_2} - \Com{a_1}{b_1}{a_4}{b_4}{a_2}{b_2} 
+ \Com{a_4}{b_4}{a_1}{b_1}{a_2}{b_2}\\
&& -\IHXthree{a_2}{a_3}{a_1}
+\Com{a_2}{b_2}{a_3}{b_3}{a_1}{b_1}
-\Com{a_3}{b_3}{a_2}{b_2}{a_1}{b_1}\\
&& -\Com{a_3}{b_3}{a_4}{a_2}{b_2}{b_4}  
+ \Com{a_3}{b_3}{a_1}{a_2}{b_2}{b_1} \\
&=& \IHXthreep{a_1}{a_4}{a_2}{a_3} 
+ \IHXthree{a_1}{a_4}{a_2} -\IHXthree{a_2}{a_3}{a_1}\\
&& + \Diff{a_3}{b_3}{a_4}{a_2}{a_1}{b_1}
+ \Diff{a_2}{b_2}{a_3}{b_4}{a_1}{b_1}+\Diff{a_3}{b_3}{a_1}{a_4}{a_2}{b_2}.
\end{eqnarray*}
From the previous paragraph, 
we know that $\IHXthree{a_1}{a_2}{a_3}$ belongs to $\ker(J)$:
so, by the $G$-action, 
$\IHXthree{a_1}{a_4}{a_2}$ and $\IHXthree{a_2}{a_3}{a_1}$ do too. 
Then, it follows from Lemma \ref{lem:D} and the above decomposition of $E_2$ that 
$\IHXthreep{a_1}{a_4}{a_2}{a_3}$ is in $\ker(J)$:
using the $G$-action, so is $\IHXthreep{a_1}{a_2}{a_3}{a_4}$.
\end{proof}

\section{Structure of certain subquotients of the Torelli group}

\label{sec:groups}

In this section, we derive from Theorem~A and Theorem~B
some explicit descriptions of certain quotients of (subgroups of) 
the Torelli group $\calI$.

\subsection{The groups $\Gamma_2 \calI/\Gamma_3 \calI$,
$\calK/\Gamma_3 \calI$ and $\calI/\Gamma_3 \calI$.}

Set $\mathrm{D}'_2(H) := S^2(\Lambda^2 H)/\Lambda^4 H$, and view it as a lattice in ${S^2(\Lambda^2 H^\Q)}/{\Lambda^4 H^\Q}
\simeq \mathrm{D}'_2(H)  \otimes \Q$.
Let also $\mathrm{D}_2(H)$
be the submodule of ${S^2(\Lambda^2 H^\Q)}/{\Lambda^4 H^\Q}$
generated by $\mathrm{D}'_2(H)$
and the elements $\frac{1}{2}(u\wedge v)\cdot(u\wedge v)$
for all $u,v\in H$.
According to \cite{Faes23}, the images 
$\tau_2([\calI,\calI])$ and $\tau_2(\calK)$ 
of the second Johnson homomorphism are determined 
by the following commutative diagram with exact rows:
\begin{equation} \label{eq:traces}
\xymatrix{
[\calI,\calI] \ar@{^{(}->}[d] \ar[r]^-{\tau_2} & 
\mathrm{D}'_2(H) \ar@{^{(}->}[d]
\ar[r]^-{\operatorname{Tr}^S} & 
\ker\big(\omega :S^2H \otimes \Z_2 \to \Z_2\big) \ar[r] \ar[d] & 0 \\
\calK \ar[r]^-{\tau_2} & \mathrm{D}_2(H) 
\ar[r]^-{\operatorname{Tr}^\Lambda} & 
\ker\big(\omega :\Lambda^2H \otimes \Z_2 \to \Z_2\big) \ar[r] & 0
}  
\end{equation}
Here, the map $S^2H \otimes \Z_2 \to \Lambda^2 H \otimes \Z_2$
is the canonical projection, and the ``trace'' maps are defined as follows:
\begin{eqnarray*}
{\operatorname{Tr}^S}\big( (a \wedge b) \cdot (c\wedge d) \big)
&:=& \omega(a,c)\, b\cdot d + \omega(a,d)\, b\cdot c
+   \omega(b,c)\, a\cdot d + \omega(b,d)\, a\cdot c, \\
{\operatorname{Tr}^\Lambda}\big( (a \wedge b) \cdot (c\wedge d) \big)
&:=& \omega(a,c)\, b\wedge d + \omega(a,d)\, b\wedge c
+   \omega(b,c)\, a\wedge d + \omega(b,d)\, a\wedge c, \\
{\operatorname{Tr}^\Lambda}\Big(\frac{1}{2}(u\wedge v)\cdot(u\wedge v)\Big)
&:=& (1+\omega(u,v))\ u\wedge v.
\end{eqnarray*}

Consider a basis $S=~\{a_1,\dots,a_g,b_1,\dots,b_g\}$
of $H$ of the form \eqref{eq:symplectic_basis}, 
and let $\langle -,-\rangle=\langle -,-\rangle_S : H \times H \to \Z$ be the symmetric bilinear form
defined by
$
\langle a_i,b_j\rangle  = \delta_{i,j},
\  \langle a_i,a_j\rangle =0, \ 
\langle b_i,b_j\rangle =0.
$
Let also $\Theta=\Theta_S: {S^2(\Lambda^2 H^\Q)}/{\Lambda^4 H^\Q} \to \Q$ 
be the map defined by 
$$
\Theta\big((a\wedge b) \cdot (c \wedge d)\big)
= \langle a,d\rangle\, \langle b,c\rangle 
- \langle a,c\rangle\, \langle b,d\rangle. 
$$

\begin{lemma}
The restriction of $\Theta$ (respectively, $2\Theta$) 
to $\ker(\operatorname{Tr}^S)$ 
(respectively, to $\ker(\operatorname{Tr}^\Lambda)$)
reduced modulo $4\Z$ does not depend on the choice of the symplectic basis $S$.
\end{lemma}

\begin{proof}
Since $\calK/[\calI,\calI]$ is $2$-torsion, the abelian group $\ker(\operatorname{Tr}^\Lambda)/\ker(\operatorname{Tr}^S)$ has to be also $2$-torsion, so it suffices to prove the statement
for the restriction of $\Theta$ to $\ker(\operatorname{Tr}^S)$.
Since $\ker(\operatorname{Tr}^S)=\tau_2([\calI,\calI])$
is the image of the map $B^{(0)}: 
\Lambda^2 (\Lambda^3H) \to \mathrm{D}'_2(H)$
defined by \eqref{eq:B0}, it suffices to check that the reduction of
$$
\theta:=\Theta\Big(\sum_{i,j\in \Z_3}
\omega(x_i,y_j) \, (x_{i+1} \wedge x_{i+2})
\cdot  (y_{j+1} \wedge y_{j+2})\Big) 
$$
modulo 4 is independent of the choice of $S$, 
for any $x_1,x_2,x_3,y_1,y_2,y_3 \in H$.
Indeed, this integer
$$
\theta=  \sum_{i,j\in \Z_3}
\omega(x_i,y_j) \, \langle x_{i+1}, y_{j+2}\rangle \,
\, \langle x_{i+2} , y_{j+1}\rangle 
-  \sum_{i,j\in \Z_3}
\omega(x_i,y_j) \, \langle x_{i+1}, y_{j+1}\rangle \,
\, \langle x_{i+2} , y_{j+2}\rangle 
$$
should be compared to the integer
\begin{eqnarray*}
\widetilde{\theta} &:=&  \sum_{i,j\in \Z_3}
\omega(x_i,y_j) \, \big( \omega(x_{i+1}, y_{j+2}) \,
 \omega( x_{i+2} , y_{j+1}) 
-  \omega( x_{i+1}, y_{j+1})
\, \omega( x_{i+2} , y_{j+2})\big) \\
&=& -3 \begin{vmatrix}
\omega(x_1,y_1) &   \omega(x_1,y_2) & \omega(x_1,y_3)\\
\omega(x_2,y_1) &   \omega(x_2,y_2) & \omega(x_2,y_3)\\
\omega(x_3,y_1) &   \omega(x_3,y_2) & \omega(x_3,y_3)\\
\end{vmatrix}  
\end{eqnarray*}
which (obviously) does not depend on $S$. 
By decomposing  every factor  in $\theta$ of the form 
$\langle x,y \rangle$ as a sum $\omega(x,y)+2\, z(x,y)$ where  $z(x,y) \in \Z$,
and distributing all the products, 
we easily observe that $\theta-\widetilde{\theta}$ belongs to $4\Z$. 
\end{proof}

We can now give a complete description 
of the abelian groups 
$\Gamma_2 \calI/ \Gamma_3 \calI$ and  $\calK/ \Gamma_3 \calI$.

\begin{theorem} \label{th:groups}
(i) The homomorphism $(\tau_2,8d'')$ induces an isomorphism
\begin{equation} \label{eq:U'}
\Gamma_2 \calI/ \Gamma_3 \calI  \stackrel{\simeq}{\longrightarrow} \mathrm{U}'(H)
\end{equation}
onto the free abelian group 
$
\mathrm{U}'(H):=\big\{ (T,z) \in \mathrm{D}'_2(H) \oplus \Z :
\operatorname{Tr}^S(T)=0, \ 
2\Theta(T) \equiv  -z\!\! \mod 8 \big\}.
$\\[0.1cm]
(ii) Similarly, the homomorphism $(\tau_2,8d'')$ induces an isomorphism
\begin{equation} \label{eq:U}
\calK/ \Gamma_3 \calI  \stackrel{\simeq}{\longrightarrow} \mathrm{U}(H)
\end{equation}
onto the free abelian group 
$
\mathrm{U}(H) := \big\{ (T,z) \in \mathrm{D}_2(H) \oplus \Z :
\operatorname{Tr}^\Lambda(T)=0, \ 
2\Theta(T) \equiv  - z\!\! \mod 4 \big\}.
$
\end{theorem}

\begin{proof}
According to Proposition  \ref{prop:invariants},
we have the commutative diagram
$$
\xymatrix{
\frac{\Lambda^2(\Lambda^3H)}{R_2} \ar@{->}[r]^-{J_2}_-\simeq \ar[rd]_-B &
\frac{\Gamma_2 \calI}{\Gamma_3 \calI} \ar[d]^-{(\tau_2,d'')} \\
& 
\frac{S^2(\Lambda^2 H^\Q)}{\Lambda^4 H^\Q} \oplus \Q
}
$$
where, according to Theorem~B, the  map $B$ is injective:
it follows that $(\tau_2,8d'')$ is injective.

To specify the image of $(\tau_2,8d'')$, 
we now review the way $d''$ can be defined 
from $\lambda$, the Casson invariant of $\Z$-homology $3$-spheres. 
Consider the Heegaard splitting $S^3=A \cup B $ of genus $g$ of the $3$-sphere,
and let $j:\Sigma \hookrightarrow S^3$ be an embedding of our surface $\Sigma$
such that  $j(\Sigma)$  is  the Heegaard surface $A\cap B=\partial A =-\partial B$ 
deprived of a  small disk $D$.
Then,   consider the map
$$
\lambda_j: \calI \longrightarrow \Z, \ 
f \longmapsto \lambda\big(S^3(f,j)\big)
$$
where $S^3(f,j)$ is the homology $3$-sphere $A \cup_f B$ obtained by ``cutting''
$S^3$ along its Heegaard surface and gluing back with 
$(j \circ f \circ j^{-1})\cup \id_D: \partial B \to \partial A$.
Then, by \cite[Lemma 7.2]{MM13}, we have
\begin{equation} \label{eq:d''}
d''= -\frac{\lambda_j}{2} - \frac{\Theta_S\circ \tau_2}{4}.
\end{equation}
Here the basis $S$ of $H$ is choosen so that $a_1,\dots,a_g$
(respectively, $b_1,\dots,b_g$) are the homology classes of a system 
of meridional curves $\alpha_1,\dots,\alpha_g$ for the handlebody $A$
(respectively, meridional curves $\beta_1,\dots,\beta_g$ for the handlebody $B$).

Let $f\in \Gamma_2 \calI$:
then $\lambda_j(f)$ is even as a consequence of \cite[Cor. 4.4]{Mor91}. 
Thus,  it follows from \eqref{eq:d''}
that $8d''(f)\equiv - 2\Theta(\tau_2(f))\!\! \mod 8$;
hence $(\tau_2(f),8d''(f))$ belongs to $\mathrm{U}'(H)$.
Conversely, let $(T,z) \in \mathrm{U}'(H)$: we claim that 
$(T,z)$ is in the image of $\Gamma_2\calI$ by $(\tau_2,8d'')$.
Indeed, by \eqref{eq:traces}, there exists $f\in \Gamma_2 \calI$ such that
$\tau_2(f)=T$. We know by \cite[Prop.~6.4]{Mor89}
(see also \cite[Prop.~6.3]{MM13}) the existence of a 
$\psi \in \ker\tau_2$ such that $\lambda_j(\psi)=1$:
hence  we have $\psi^2 \in \Gamma_2 \calI  \cap \ker\tau_2$
and $\lambda_j(\psi^2)=2$.
Therefore, for any $\ell\in \Z$,
the element $f_\ell := f \psi^{2\ell}$ of $\Gamma_2 \calI$ satisfies 
$$
\tau_2(f_\ell) = \tau_2(f)=T
\quad \hbox{and} \quad 
d''(f_\ell) \stackrel{\eqref{eq:d''}}{=} d''(f)-\ell.
$$
Thus, taking $\ell:= d''(f)-z/8 \in \Z$,
we obtain that the value of $f_\ell$ by $(\tau_2,8d'')$ is $(T,z)$.
This achieves the proof of (i): the homomorphism \eqref{eq:U'} is bijective.

We now prove (ii). 
The surjectivity of \eqref{eq:U} is proved similarly to the surjectivity of \eqref{eq:U'}.
Moreover, the injectivity of \eqref{eq:U'} implies the injectivity of \eqref{eq:U} in the following way.
Let $k\in \calK$ be such that $\tau_2(k)=0$ and $d''(k)=0$.
The fact that $\tau_2(k)=0$ implies that the function $\beta(k):\mathcal{Q} \to \Z_2$
(image of $k$ by the Birman--Craggs homomorphism) is constant \cite[p$.$ 178]{Jo83},
and the fact that $d''(k)=0$ implies then that $\beta(k)$ is  null.
We deduce from Johnson's determination~\eqref{eq:ab_Torelli} 
of $\calI_{\operatorname{ab}}$ that $k\in \Gamma_2 \calI$. 
\end{proof}

The following result, 
which describes the structure of the group $\calI/\Gamma_3 \calI$,
is a refinement of Morita's description \cite[Th.~3.1]{Mor91}
of the group $\calI/\ker(\tau_2)$.

\begin{theorem} \label{th:I/Gamma3I}
The  extension of groups 
$$
\xymatrix{
1 \ar[r] & \calK/\Gamma_3 \calI \ar[r]
& \calI/\Gamma_3 \calI \ar[r]^-{\tau_1} 
& \Lambda^3 H \ar[r] & 1
}
$$
is central, and its characteristic class
viewed as an element of 
$$
H^2\big(\Lambda^3 H;  \calK/\Gamma_3 \calI\big) 
\simeq \Hom\big(\Lambda^2(\Lambda^3 H), \calK/\Gamma_3 \calI\big) 
\stackrel{\eqref{eq:U}}{\simeq} \Hom\big(\Lambda^2(\Lambda^3 H), \mathrm{U}(H)\big) 
$$
maps any 
$(x_1\wedge x_2 \wedge x_3) \wedge (y_1\wedge y_2 \wedge y_3)$
to
$$
\Bigg( \sum_{i,j\in \Z_3}
\omega(x_i,y_j) \, (x_{i+1} \wedge x_{i+2})
\cdot  (y_{j+1} \wedge y_{j+2}),
- 2 \begin{vmatrix}
\omega(x_1,y_1) &   \omega(x_1,y_2) & \omega(x_1,y_3)\\
\omega(x_2,y_1) &   \omega(x_2,y_2) & \omega(x_2,y_3)\\
\omega(x_3,y_1) &   \omega(x_3,y_2) & \omega(x_3,y_3)\\
\end{vmatrix}  \Bigg).
$$
\end{theorem}

\begin{proof}
The centrality of the extension is given by Theorem~A.
Let $s: \Lambda^3 H \to \calI/\Gamma_3 \calI$
be a setwise section of $\tau_1$, 
and let $c: \Lambda^3H \times \Lambda^3H \to \calK/\Gamma_3\calI $ 
be the corresponding $2$-cocycle, 
i.e$.$ $c(x\vert y)=s(x)\, s(y)\, s(xy)^{-1}$.
Then, as an element of 
$\Hom(\Lambda^2(\Lambda^3 H), \calK/\Gamma_3\calI)$,
the cohomology class~$C$ of $c$ is given by
$
C(x\wedge y) = c(x\vert y)\, c(y \vert x)^{-1}
= [s(x),s(y)].
$
Thus, 
we have a commutative diagram
$$
\xymatrix{
\calI/\calK \times \calI/\calK \ar[r]^{[-,-]} 
\ar[d]_-{\tau_1 \times \tau_1}^-\simeq 
&  \Gamma_2 \calI / \Gamma_3 \calI  \ar@{^{(}->}[d] \\
\Lambda^3H \times \Lambda^3H \ar[r]^-C &  \calK/\Gamma_3\calI
}
$$
where the composite $[-,-] \circ (\tau_1 \times \tau_1)^{-1}$
is the map $J_2$.
Then the conclusion comes from Proposition~\ref{prop:invariants}
and the definition of the map $B$.
\end{proof}

We now give three applications of the previous theorems.

\begin{corollary} \label{cor:I/Gamma3I}
The group $\calI/\Gamma_3 \calI$ is torsion-free,
and its center is $\calK/\Gamma_3\calI$.
\end{corollary}

\begin{proof}
The group $\calI/\Gamma_3 \calI$ is torsion-free 
because (by Theorem~\ref{th:groups}\,(ii)) it is an extension
of a torsion-free abelian group by another one.

Let $f\in \calI$ with the property that
 $[f,h]\in \Gamma_3 \calI$ for all $h\in \calI$. 
 Then, we have $d''([f,h])=0$ for all $h\in \calI$
and, by Proposition \ref{prop:invariants}, we obtain that $\tau_1(f) \in \Lambda^3H$ belongs 
to the kernel of the bilinear map
$$
b: \Lambda^3 H \times \Lambda^3 H  \longrightarrow \Z, \ 
(x_1\wedge x_2 \wedge x_3, y_1\wedge y_2 \wedge y_3) \longmapsto
\begin{vmatrix}
\omega(x_1,y_1) &   \omega(x_1,y_2) & \omega(x_1,y_3)\\
\omega(x_2,y_1) &   \omega(x_2,y_2) & \omega(x_2,y_3)\\
\omega(x_3,y_1) &   \omega(x_3,y_2) & \omega(x_3,y_3)\\
\end{vmatrix}.
$$
But, this form is non-singular since
$b(s_1 \wedge s_2 \wedge s_3, \ov{s_1} \wedge \ov{s_2} \wedge \ov{s_3}) = \pm 1$
for any pairwise-different $s_1,s_2,s_3 \in S$. 
Therefore, we have $\tau_1(f)=0$ so that $f\in \calK$.
This proves that the center of $\calI/\Gamma_3 \calI$ 
is contained in $\calK/\Gamma_3 \calI$. The converse inclusion is given by Theorem~\ref{th:I/Gamma3I}.
\end{proof}

Denote by 
$\calM=\calM[0]  \supset \calM[1]  \supset \calM[2] \supset \cdots$
the \emph{Johnson filtration} of the mapping class group~$\calM$. 
Thus, for every $k\geq 0$, $\calM[k]$ is the subgroup of $\calM$ acting trivially on the $k$-th nilpotent quotient
$\pi/\Gamma_{k+1} \pi$ of the fundamental group 
$\pi=\pi_1(\Sigma, \star)$.
So far, we have encountered
$\calM[1]=\calI$, 
$\calM[2]=\calK=\ker \tau_1$ and $\calM[3]=\ker \tau_2$.
The inclusion $\Gamma_k \calI \subset \calM[k]$ 
holds true for every $k\geq 1$,
and determining the gap between  
the two filtrations  is an important question: thus,
the  problem is to determine the kernel and the image
of the induced homomorphism
\begin{equation}  \label{eq:graded_bis}
 {\Gamma_k \calI}/{\Gamma_{k+1} \calI}
\longrightarrow   {\calM[k]}/{\calM[k+1]}.
\end{equation}

In degree $k=1$,
the kernel  $\calM[2]/\Gamma_2 \calI$
of \eqref{eq:graded_bis}
is isomorphic to the $2$-torsion group $B_{\leq 2}(\mathcal{Q})$
via the Birman--Craggs homomorphism $\beta$. 
The following, which identifies the kernel of \eqref{eq:graded_bis}
for  $k=2$, is a direct consequence 
of Theorem~\ref{th:groups}\,(i).

\begin{corollary} \label{cor:intersection}
We have 
$(\Gamma_2\calI \cap \calM[3])/\Gamma_3 \calI \simeq \Z$
via $d''$.
\end{corollary}

By \cite{Jo80},
the map \eqref{eq:graded_bis}  is surjective in degree $1$;
but it is not in degree $2$ since $\beta$ induces 
an isomorphism between $\calM[2]/(\calM[3] \cdot \Gamma_2 \calI)$
and $B_{\leq 2}(\mathcal{Q})/B_{\leq 0}(\mathcal{Q})$ \cite[Prop$.$~3.3]{Yok02}.
The following gives the surjectivity
of  \eqref{eq:graded_bis} in degree~$3$.

\begin{corollary} \label{cor:M[3]}
We have $\calM[3] = \calM[4] \cdot \Gamma_3 \calI$.
\end{corollary}

\begin{proof}
    Let $f \in \calM[3]$. By \cite[Th$.$ A]{Faes22}, 
    there exists an element $\psi \in \calM[4]$ 
    such that $\lambda_j(\psi) = -1$, 
    which implies that $2d''(\psi) = 1$ by \eqref{eq:d''}. 
    We also have that $2d''(\calM[3]) \subset \mathbb{Z}$, 
    hence setting $k := \psi^{2d''(f)} \in \calM[4]$, 
    we obtain that $fk^{-1} \in \calM[3]$ and $d''(fk^{-1}) = 0$.
    It follows from  Theorem~\ref{th:groups}\,(ii) that
    $fk^{-1}$ belongs to $\Gamma_3 \calI$.
\end{proof}

\subsection{Relation to the monoid of homology cylinders}

Let $\mathcal{IC}:=\mathcal{IC}(\Sigma)$ be the monoid of \emph{homology cylinders} over the surface $\Sigma$:
its neutral element is the trivial cylinder $\Sigma \times [-1,+1]$,
and its operation is the ``stacking product''.
The reader is referred to the survey paper \cite{HM12} for the precise definition of $\mathcal{IC}$,
and an overwiew of its relationship with the study of the Torelli group. 
Here, we simply recall a few facts and some notations about  homology cylinders:
\begin{itemize}
\item The monoid $\mathcal{IC}$ comes with a sequence of \emph{$Y_k$-equivalence} relations  (for $k\geq 1$),  
which are defined by surgery techniques.
\item The  \emph{$Y$-filtration}
$\mathcal{IC}= Y_1\mathcal{IC} \supset Y_2\mathcal{IC}  \supset Y_3\mathcal{IC} \supset \cdots$
on the monoid $\mathcal{IC}$ is defined by 
$ Y_k\mathcal{IC} := \{ M \in \mathcal{IC} : M \hbox{ is $Y_k$-equivalent to } \Sigma \times [-1,+1] \}$.
\item For every $k\geq 1$, the quotient monoid $\mathcal{IC}/Y_k$ is actually a group.
\item The ``mapping cylinder'' construction defines a monoid homomorphism 
$\textbf{c}:\calI \to \mathcal{IC}$, 
which is injective and maps the lower central series to the $Y$-filtration.
\end{itemize}

It is an important problem to compare the lower central series of $\calI$
with the pull-back by~$\textbf{c}$ of the $Y$-filtration.
In the lowest (non-trivial) degree, 
it is known that the homomorphism
$\textbf{c}:\calI/\Gamma_2\calI \to \mathcal{IC}/Y_2$
induced by $\textbf{c}$
is an isomorphism \cite{Ha00,MM03}.
In the next degree,  the group $\mathcal{IC}/Y_3$
has been given a complete description in \cite[Th$.$ 5.8]{MM13},
in the same spirit of Theorem~\ref{th:I/Gamma3I} 
and in terms of ``symplectic Jacobi diagrams'';
but the question of its comparison 
with the group $\calI/\Gamma_3 \calI$ was left  open.
We can now answer this question.

\begin{theorem} \label{th:IC/Y3}
There is an exact sequence of groups
\begin{equation} \label{eq:sequence}
\xymatrix @C=1.1cm  {
1 \ar[r] & \calI/\Gamma_3\calI  \ar[r]^-{\textbf{c}} 
& \mathcal{IC}/Y_3 \ar[r]^-\alpha & 
S^2H \ar[r]^-{\omega \!\!\!\!\mod 2} & \Z_2 \ar[r] & 1
}
\end{equation}
where $\alpha$ 
denotes the ``quadratic part''
of the relative Alexander polynomial of homology cylinders, 
as defined in \cite[\S 3.2]{MM13}.
\end{theorem}

\begin{proof} 
Some of the homomorphisms 
that we considered in the previous subsections
on (subgroups of) $\calI$
extend to (submonoids of) $\mathcal{IC}$.
Specifically, we shall need the following:
\begin{itemize}
\item The first Johnson homomorphism extends to $\tau_1: \mathcal{IC} \to \Lambda^3 H$,
and encodes the action of $\mathcal{IC}$ on $\pi/\Gamma_3 \pi$;
we denote its kernel by $ \mathcal{KC}:=  \mathcal{KC}(\Sigma)$.
\item The second Johnson homomorphism extends to 
$\tau_2: \mathcal{KC} \to \mathrm{D}_2(H)$,
and encodes the action of $\mathcal{KC}$ on $\pi/\Gamma_4 \pi$.
\item The second core of the Casson invariant extends to $d'': \mathcal{KC} \to \Q$,
and formula \eqref{eq:d''} generalizes to\\[-0.8cm]
\begin{equation} \label{eq:d''_bis}
d''= -\frac{\lambda_j}{2} - \frac{\check{\Theta}_S\circ \alpha}{4}
- \frac{\Theta_S\circ \tau_2}{4},
\end{equation}
where $\lambda_j: \mathcal{IC} \to \Z$ measures 
the Casson invariant after the ``insertion'' of homology cylinders,
and $\check{\Theta}_S: S^2H \to \Q$
is a certain homomorphism depending on the choice of the basis $S$ of $H$.
\end{itemize}

We prove the injectivity of $\textbf{c}:\calI/\Gamma_3\calI \to \mathcal{IC}/Y_3$.
Let $f\in \calI$ whose mapping cylinder
$\textbf{c}(f)$ is $Y_3$-equivalent to the trivial cylinder.
Since the action of $\mathcal{IC}$
on $\pi/\Gamma_4 \pi$ is preserved by $Y_3$-surgery,
we have $f\in \calK$ and $\tau_2(f)=0$.
Since the Casson invariant is preserved by $Y_3$-surgery,
we also have $d''(f)=0$. 
It follows from the injectivity of \eqref{eq:U} in Theorem~\ref{th:groups}
that $f$ belongs to $\Gamma_3 \calI$.

That the image of $\textbf{c}:\calI/\Gamma_3\calI \to \mathcal{IC}/Y_3$ is contained
in $\ker \alpha$ is given by \cite[Prop. 3.13]{MM13}.
To prove the converse inclusion,
let $M\in \mathcal{IC}$ satisfying $\alpha(M)=0$.
Since $\tau_1: \calI \to \Lambda^3 H$ is surjective,  
we observe that  $\mathcal{IC}=\mathcal{KC} \cdot \mathbf{c}(\calI)$:
thus, there is an $f\in \calI$ such that 
$M_0:= M \cdot \mathbf{c}(f^{-1})$ 
belongs to $\mathcal{KC}$.
Next, according to \cite[Prop. 6.1]{Faes23pre},  
we have the  commutative diagram
\begin{equation} \label{eq:square}
\xymatrix{
\mathcal{KC} \ar[r]^-\alpha \ar[d]_-{\tau_2} & 
S^2 H \ar@{->>}[d]^-p \\
\mathrm{D}_2(H)  \ar[r]_-{\operatorname{Tr}^\Lambda}  
& \Lambda^2 H\otimes \Z_2
}    
\end{equation}
where the projection $p$
maps $h\cdot k$ to $(h\wedge k \!\! \mod 2)$ for all $h,k\in H$.
Therefore, the fact that $\alpha(M_0)=\alpha(M)=0$
implies that $\tau_2(M_0) \in \ker \mathrm{Tr}^\Lambda $; besides, 
we deduce from \eqref{eq:d''_bis} that 
$8d''(M_0)\equiv -2 \Theta(\tau_2(M_0))\!\!\mod 4$. Then, it follows from the surjectivity of \eqref{eq:U} in Theorem~\ref{th:groups}
that there exists   $u \in \calK$ such that $d''(u)=d''(M_0)$
and $\tau_2(u)=\tau_2(M_0)$.
As a consequence of \cite[Th$.$~A]{MM13}, we deduce that $\mathbf{c}(u)$
is $Y_3$-equivalent to $M_0$. Hence we obtain that 
$$
\{M\}= \big\{M_0\, \mathbf{c}(f) \big\} 
= \big\{\mathbf{c}(uf) \big\}  \ \in \mathcal{IC}/Y_3
$$
belongs to the image of $\mathbf{c}$.

Finally, the sequence \eqref{eq:sequence} is also exact at $S^2 H$ 
for the following reasons.
The identity
$\mathcal{IC}=\mathcal{KC} \cdot \mathbf{c}(\calI)$
implies that $\alpha(\mathcal{IC})=\alpha(\mathcal{KC})$.
We know from \cite[Th$.$~3]{GL05} 
that the map $\tau_2$ in \eqref{eq:square} is surjective,
and we deduce from \cite[Prop. 3.13]{MM13} that $\alpha(\mathcal{KC})$ contains 
$\ker p=\langle h^2 \,\vert\, h\in H \rangle$.
Thus, we conclude that
\begin{eqnarray*}
\alpha(\mathcal{KC}) 
\ = \ \alpha(\mathcal{KC}) + \ker p
&=& p^{-1}\big( \operatorname{Tr}^\Lambda(\mathrm{D}_2(H))\big)\\
&\stackrel{\eqref{eq:traces}}{=}&
p^{-1}\big( \ker \omega: \Lambda^2 H \otimes \Z_2 \to \Z_2\big)
\ = \  \big( \ker (\omega\!\!\!\mod 2): S^2 H  \to \Z_2\big).
\end{eqnarray*}
\end{proof}

\end{document}